\documentclass[12pt]{article}

\usepackage{graphicx}
\usepackage{latexsym}
\usepackage{color}
\usepackage{amssymb,amsmath}

\catcode `\@=11
\@addtoreset{equation}{section}

\catcode`\@=12

\newtheorem{theorem}{Theorem}[section]
\newtheorem{def.}[theorem]{Definition}
\newtheorem{prop.}[theorem]{Proposition}
\newtheorem{lem.}[theorem]{Lemma}
\newtheorem{cor.}[theorem]{Corollary}
\newtheorem{conj.}[theorem]{Conjecture}
\newtheorem{Bsp.}{Example}[section]
\newtheorem{rem.}{Remarks}[theorem]
\newtheorem{rem}{Remark}[theorem]

\newenvironment{proof}{\noindent \bf Proof: \rm}{$ \hspace{\stretch{1}} \Box $}

 \newcommand{\bea}{\begin{eqnarray}}
\newcommand{\ena}{\end{eqnarray}}

\newcommand{\beano}{\begin{eqnarray*}}
\newcommand{\enano}{\end{eqnarray*}}
 
\newcommand{\bei}{\begin{itemize}}
\newcommand{\eni}{\end{itemize}}

\newcommand{\be}{\begin{equation}}
\newcommand{\en}{\end{equation}}

\newcommand{\belem}{\begin{lem.}}
\newcommand{\enlem}{\end{lem.}}

\newcommand{\berem}{\begin{rem} \rm}
\newcommand{\enrem}{\end{rem}}

\newcommand{\norm}[2]{
\left\| #2 \right\|_{#1}
}
\newcommand{\Hil}[0]{
\mathcal{H} 
}
\newcommand{\h}[0]{
\mathcal{H} 
}

\newcommand{\CC}[0]{
{\mathbb C}
}

\renewcommand{\le}{\leqslant}
\renewcommand{\leq}{\leqslant}
\renewcommand{\geq}{\geqslant}

\newcommand{\nn}{\nonumber}
 \newcommand{\noi}{\noindent}
\newcommand{\dis}{\displaystyle}

 \newcommand{\ov}{\overline}
\newcommand{\RN}{\mathbb{R}}

\newcommand{\CN}{{\mathbb C}}

\newcommand{\ip}[2]{\langle {#1},{#2}\rangle}

\newcommand{\dom}{{\sf Dom}}
\newcommand{\ran}{{\sf Ran}}

\def\h{{\mathfrak H}}
\def\s{{\mathfrak S}}

\newcommand{\NN}[0]{\mathbb{N}}
\newcommand{\ud}{\,\mathrm{d}}
\newcommand{\up}{\raisebox{0.7mm}{$\upharpoonright$}} 

\def\<{\langle}
\def\>{\rangle}

\begin{document}
\title{Frames and Semi-Frames}
\author{Jean-Pierre Antoine\thanks{Institut de Recherche en Math\'ematique et  Physique, Universit\'e catholique de Louvain,
B - 1348   Louvain-la-Neuve, Belgium; jean-pierre.antoine@uclouvain.be}
\, and Peter Balazs\thanks{Acoustics Research Institute, Austrian Academy of Sciences, Wohllebengasse 12-14,  
A-1040 Vienna, Austria; peter.balazs@oeaw.ac.at}
}
\maketitle

 \begin{abstract} 
  Loosely speaking, a semi-frame is a generalized frame for which one of the frame bounds is absent.
More precisely, given a total sequence in a Hilbert space, 
we speak of an upper (resp. lower) semi-frame if  only the upper (resp. lower)  frame bound is valid. 
Equivalently, for an  upper  semi-frame, the frame operator is bounded, but has an unbounded inverse, 
whereas  a lower  semi-frame  has an unbounded frame operator, with bounded inverse.
  We study mostly upper  semi-frames, both in the continuous case and in the discrete case, and give some remarks for the dual situation.
In particular, we show that reconstruction is still possible in certain cases.
\end{abstract}

\noi PACS numbers: 02.30.Lt, 02.30.Mv, 02.70.-c, 02.90.+p

\section{Introduction}
\label{sec:intro}

The notion of frame was introduced in 1952 by Duffin and Schaefer \cite{duffschaef1} in the context of nonharmonic analysis. It was revived by 
Daubechies, Grossmann and Meyer \cite{daub-painless} in the early stages of wavelet theory and then became a very popular topic,
 in particular in Gabor and wavelet analysis  \cite{daubech1,Casaz1,ole1,Groech1}. The reason is that a good frame in a Hilbert space is almost as good as an orthonormal 
basis for expanding arbitrary elements (albeit non-uniquely) and is often easier to construct.  
 In order to put the present work in perspective, we recall that  a sequence $\Psi = (\psi_k)$ is a frame for a Hilbert space
$\Hil$ if there exist constants  ${\sf m}>0$ and ${\sf M}<\infty$ such that
\be
{\rm\sf m} \norm{}{f}^2 \le  \sum _{k\in \Gamma} \left| \ip {\psi_k}{f} \right|^2  \le {\rm\sf M}  \norm{}{f}^2 ,  \forall \, f \in \Hil.
\label{eq:discr-bddframe}
\end{equation}
 
 All the above concerns sequences, as required in numerical analysis. However, in the meantime, more general objects, called continuous frames, 
emerged in the context of the theory of (generalized) coherent states  and were thoroughly studied by Ali, Gazeau and one of us   
  \cite{jpa-sqintegI,jpa-contframes,jpa-CSbook} (they were introduced independently by Kaiser \cite{kaiser}).
They were studied further by Askari-Hemmat  \emph{et al.} \cite{askari},  Gabardo and Han \cite{gabardo-han},  Rahimi \emph{et al.} \cite{rahimi}, 
 Fornasier and Rauhut \cite{forn-rauhut}, 
and by Zakharova \cite{zakharova} (under the name `integral frames').
 Actually the discrete frames are just a particular case. In that respect, in particular, Fornasier and Rauhut study in great detail the problem of discretization of a continuous frame, by sampling 
the base space $X$.
 
Before proceeding, we list our definitions and  conventions. The framework is
 a (separable) Hilbert space $\Hil$, with the inner product $\ip{\cdot}{\cdot}$ linear in the second factor.
Given an operator $A$ on $\Hil $, we denote its domain by $\dom(A)$, its range by $\ran(A)$  or, shorter, $R_A$, and its kernel by ${\sf Ker}(A)$.
 An operator $A$ in $\Hil$ is called \emph{positive}, resp.  \emph{nonnegative}, if $\ip{h}{Ah}> 0$,  resp. $\ip{h}{Ah} \geq 0, \, \forall\, h\neq 0, h\in \dom(A)$.
We call an operator $A$ \emph{invertible}, if it is invertible as a function from $\dom(A)$ to $\ran(A)$, i.e. if it is injective. 
$GL(\Hil)$ denotes the set of all bounded operators on $\Hil$ with bounded inverse.

Let   $X$ be a  locally compact space with measure $\nu$. We assume that $X$ is    $\sigma$-compact, that is,
 $X = \bigcup_{n} K_{n}, K_{n} \subset  K_{n+1}, K_{j}$ relatively compact. 
Let  $\Psi:=\{\psi_{x},\, x\in X\}$ be a family of vectors  
from a Hilbert space $\Hil$ indexed by points of $X$. 
Then we say that  $\Psi$ is a set of coherent states    or a \emph{generalized frame} if the map  $x \mapsto \ip{f}{\psi_{x}}$  
is measurable for all $f \in \Hil$ and
$$
\int_{X}\ip{f}{\psi_{x}}\ip {\psi_{x}}{f'}\, \ud\nu(x) = \ip{f}{Sf'}, \; \forall\, f, f' \in \Hil,
$$
where $S$ is a bounded, positive, self-adjoint, invertible operator on $\Hil $, called the frame operator
\cite{jpa-sqintegI}-\cite{jpa-CSbook} (several (mathematical) authors \cite{gabardo-han, rahimi}
  call \emph{frame} the map $\breve\psi: X\to \Hil$ given as $ \breve\psi (x) =  \psi_{x}$). 
 In Dirac's notation, the frame operator $S$ reads
$$
S = \int_{X}|\psi_{x}\rangle \langle\psi_{x}| \ud x.
$$

The operator $S$ is invertible, but its inverse $S^{-1}$, while still self-adjoint and positive, need not be bounded.
Thus, we say that   $\Psi$ is a \emph{frame} if $ S^{-1}$ is bounded or, equivalently,
if there exist constants  ${\sf m} > 0$  and ${\sf M}<\infty$  (the  frame bounds) such that 
\be\label{eq:frame}
{\sf m}  \norm{}{f}^2 \le \ip{f}{Sf} = \int_{X}  |\ip{\psi_{x}}{f}| ^2 \, \ud \nu(x) \le {\sf M}  \norm{}{f}^2 ,  \forall \, f \in \Hil.
\end{equation}
 For frames the  spectrum {\sf Sp}$(S)$ of $S$   is contained in the interval $ [{ \sf m}, {\sf M}]$, 
these two numbers being the infimum and the supremum of {\sf Sp}$(S)$, respectively.

These definitions are completely general. In particular, if $X$ is a discrete set with $\nu$ being a counting measure, 
we recover the standard definition \eqref{eq:discr-bddframe} of a (discrete) frame \cite{duffschaef1,Casaz1,ole1}.
In that case too, one defines the frame operator $S$ which is bounded, self-adjoint, positive, invertible, and so is its inverse $S^{-1}$. 
However, there are situations where the notion of frame is too restrictive, in the sense that one cannot satisfy \emph{both} frame bounds simultaneously.
Thus there is room for two natural generalizations. We will say that
a family  $\Psi$ is   an \emph{upper (resp. lower) semi-frame}, if
\bei
\vspace*{-1mm}\item[(i)]  {it is total in $\Hil$;}
\vspace*{-2mm}\item[(ii)] {it satisfies the upper (resp. lower) frame inequality.}
\eni
\vspace*{-1mm} Note that the lower frame inequality automatically implies  that the family is total, i.e. (ii) $\Rightarrow$ (i) for a  lower semi-frame.
Also, in the upper case, $S$   is bounded and $S^{-1}$ is unbounded, whereas, in the lower case,  $S$   is unbounded and $S^{-1}$ is bounded.
We may also remark that a discrete upper   semi-frame is nothing but a total Bessel sequence   
(in the frame community, this is called a \emph{complete} Bessel sequence, but the word is ambiguous). 
These are the notions we want to   extend to the general case.

Let us go back to  the continuous case.  If one has
\be\label{eq:upframe}
0 < \int_{X}  |\ip{\psi_{x}}{f}| ^2 \, \ud \nu(x)   \le {\rm\sf M}  \norm{}{f}^2 ,  \forall \, f \in \Hil, \, f\neq 0,
\end{equation}
then $\Psi$ is called a (continuous) \emph{upper semi-frame}. 
 In this case, $S^{-1}$ is unbounded, with dense domain $\dom(S^{-1})$.
(In the terminology of \cite{gabardo-han}, the corresponding mapping $\breve\psi$ could be called  a \emph{positive  Bessel mapping}.
 In our previous work, this object was called an \emph{unbounded frame},
but this terminology is somewhat counterintuitive, since  an unbounded frame is not a frame!)

{By symmetry (in fact, duality, as we will see below),    we  will speak of a \emph{lower semi-frame} if the upper frame bound is missing. Note that, since  $S$ may now be unbounded,
a lower semi-frame is no longer a coherent state. We will come back to these matters in Section \ref{subsec:upperlower}.}

In the present paper, we will study mostly upper  semi-frames and only give some remarks for the dual situation. In particular, we show that reconstruction is still possible.
We will first cover the general  (continuous) case, then particularize the results to the discrete case, as required in practice.

\section{Continuous frames revisited}
\label{sec:cont-frames}

\subsection{Continuous frames}
\label{subsec:bdd-frames}

In order to get a feeling for the general situation, we begin by quickly recalling the main results in the  (standard) 
case of a frame, where  $S$ and $S^{-1}$ are both bounded
\cite{jpa-sqintegI,jpa-contframes,jpa-CSbook}. Note that we use here the term ``continuous frame" which is now well established in the literature. 
However, the case envisaged here is general, since it contains the discrete case as well, when $X$ is a discrete set and $\nu$ a counting measure.
Note, however, that in the continuous case treated here, most results concern only \emph{weak} convergence, whereas in the discrete case (Section \ref{sec: discreteframes}), 
one usually requires results about \emph{strong} convergence. 

First,  $\Psi$ is a total set in $\Hil$. Next 
define the   (coherent state) map  $C_{\Psi}: \Hil \to L^{2}(X, \ud\nu)$ by  
\be
(C_{\Psi}f)(x) =\ip{\psi_{x}}{f} , \; f \in \Hil.
\label{eq:csmap}
\end{equation}
(This map was denoted $W_{\Psi}$ in the previous works 
\cite{jpa-sqintegI,jpa-contframes,jpa-CSbook}, but the present notation is closer to the one used in frame theory.)
Mimicking the terminology of frame theory, we may call $C_{\Psi}$ the \emph{analysis operator}. Its adjoint
$C_{\Psi}^\ast: L^{2}(X, \ud\nu) \to \Hil$, 
called the \emph{synthesis operator}, reads (the integral being understood in the weak sense,  as usual \cite{forn-rauhut})
\be
C_{\Psi}^\ast F =  \int_X  F(x) \,\psi_{x} \; \ud\nu(x), \mbox{ for} \;\;F\in L^{2}(X, \ud\nu)  
\label{eq:synthmap}
\end{equation}
Then $C_{\Psi}^* C_{\Psi} = S$ and $ \|C_{\Psi} f\|^2_{L^{2}(X)}= \| S^{1/2}f\|_{\Hil}^2 = \ip{f}{Sf}$.
 Furthermore,  $C_{\Psi}$ is injective, since $S>0$,  so that $ C_{\Psi}^{-1} :{ \sf Ran} (C_{\Psi}) \to \Hil$ is well-defined.

Next,  \ran$(C_{\Psi})$  is a \emph{closed} subspace of $L^{2}(X, \ud\nu )$,  as follows from the lower frame bound, which implies that 
$S^{-1}$ is bounded.
 The corresponding orthogonal projection can be computed as follows. First we define the (Moore-Penrose) pseudo-inverse of  $C_{\Psi}$, namely, 
$$
 C_{\Psi}^+ := (C_{\Psi}^* C_{\Psi})^{-1} C_{\Psi}^* = S^{-1} C_{\Psi}^*.
 $$
 This is indeed a left inverse of $C_{\Psi}$, i.e. $C_{\Psi}^+C_{\Psi} = I$.
Then the operator ${\mathbb P}_{\Psi}:  L^{2}(X, \ud\nu) \to \ran(C_{\Psi})$ defined by
$$
{\mathbb P}_{\Psi}:= C_{\Psi} S^{-1} C_{\Psi}^* = C_{\Psi}  C_{\Psi}^+ 
$$
is the orthogonal projection on \ran$(C_{\Psi})$. Indeed, it is self-adjoint and idempotent, and its range is clearly \ran$(C_{\Psi})$.
This  projection is an integral operator with (reproducing) kernel $K(x,y) = \ip{\psi_{x} }{S^{-1}\psi_{y}}$,
i.e.,  \ran$(C_{\Psi})$ is a  {reproducing kernel Hilbert space}.
The interest of this fact is that the elements of  \ran$(C_{\Psi})$ are genuine functions, not equivalence classes. 

 In addition, the subspace   \ran$(C_{\Psi})$ is also  complete in the norm $\| \cdot\|_{\Psi}$, associated to the new scalar product   defined by 
\be
\ip {F}{F'}_{\Psi} := \ip {F}{C_{\Psi} \,S^{-1} \,C_{\Psi}^{-1} F' }_{L^{2}(X)} , \mbox{ for } \; F,F'\in \ran(C_{\Psi}) .
\label{eq:scalarprod}
\end{equation}
Hence it is a Hilbert space, denoted by ${\h}_{\Psi}$. The map    $C_{\Psi}: {\Hil} \to {\h}_{\Psi}$ is unitary, since it  is both isometric and onto. 
One has indeed, for every $F,F' \in   \ran(C_{\Psi})$, 
\bea \label{eq:unitary}
\ip {F}{F'}_{\Psi} &=& \ip {C_{\Psi} f}{C_{\Psi} f'}_{\Psi}= \ip {C_{\Psi} f}{C_{\Psi} \,S^{-1} \,C_{\Psi}^{-1} C_{\Psi} f' }_{L^{2}(X,\ud\nu)} \nn
\\
&=& \ip {C_{\Psi} f}{C_{\Psi} \,S^{-1} \,  f' }_{L^{2}(X,\ud\nu)} \nn
\\
&=& \ip { f}{C_{\Psi}^\ast C_{\Psi} \,S^{-1} \,  f' }_{\Hil} \nn
\\
&=& \ip {f}{f'}_{\Hil}.
\ena 
 $C_{\Psi}$ being a unitary operator, it can be inverted  on its range by the adjoint operator
 $C_{\Psi}^{\ast (\Psi)}: {\h}_{\Psi} \to {\Hil} $,  which is nothing but the pseudo-inverse  $C_{\Psi}^+ = S^{-1} C_{\Psi}^*.$
Thus one gets,  for every $f\in \Hil$, a  {reconstruction formula}, where the integral converges weakly:
\be
f = C_{\Psi}^{-1} F = C_{\Psi}^+ F =  \int_X  F(x) \,S^{-1}\,\psi_{x} \; \ud\nu(x), \mbox{ for} \;\;F= C_{\Psi}f \in {\h}_{\Psi} .
\label{eq:reconstr}
 \end{equation}

\subsection{Continuous upper semi-frames}
\label{subsec:unbdd-frames}

Now let us look at families that satisfy \eqref{eq:upframe}, i.e., upper semi-frames. In this case, the operators $C_{\Psi}$ 
and $S$ are bounded, $S$ is injective and self-adjoint.
Therefore $R_S$ is dense in $\Hil$ and $S^{-1}$ is also self-adjoint. 
$S^{-1}$ is unbounded, with dense domain $\dom(S^{-1}) = R_S$ \cite{jpa-sqintegI,jpa-contframes,jpa-CSbook}.
Once again,  $\Psi$ is a total set in $\Hil$.

One has the following diagram, where  we write $ R_{C}:=   \ran(C_{\Psi})$
\be\label{diagram1}
\begin{array}{cccc}
\Hil       &\stackrel{C_{\Psi}}{\longrightarrow }& R_{C} \subset & \!\!\!\!\ov{R_{C}} \subset  L^{2}(X, \ud\nu) \\[1mm]
\cup&                            & \cup& \\[1mm]
\dom(S^{-1}) = R_{S}   &\stackrel{C_{\Psi}}{\longrightarrow }  &C_{\Psi}(R_{S})\subset & \hspace{-1.3cm} L^{2}(X, \ud\nu)   
\end{array}
\end{equation}
where $\ov{R_{C}}$ denotes the closure of $R_{C}$ in $ L^{2}(X, \ud\nu)  $ (indeed, $R_{C}$ need no longer be closed). 

 Define the Hilbert space ${\h}_{\Psi}:=\ov{C_{\Psi}(R_{S})}^{\Psi} $, where the  
 completion is taken with respect to the norm $\| \cdot\|_{\Psi}$.   
 defined in \eqref{eq:scalarprod}.
 Then, the same calculation as in \eqref{eq:unitary} above shows that the map $C_{\Psi}$, restricted to the dense domain  $\dom(S^{-1}) =   R_{S}$, is an  isometry
 onto  $C_{\Psi}(R_{S}) \subset {\h}_{\Psi}$ :
$$
 \ip {C_{\Psi} f}{C_{\Psi} f'}_{\Psi}=\ip {f}{f'}_{\Hil}, \; \forall\, f,g\in  R_{S}.
 $$
  Thus $C_{\Psi}$ extends by continuity to a  \emph{unitary} map between the respective completions, namely, from $\Hil$ onto ${\h}_{\Psi}:=\ov{C_{\Psi}(R_{S})}^{\Psi} $.

Therefore, $ {\h}_{\Psi} $ and $  R_{C}$ coincide as sets,
 so that  $ {\h}_{\Psi} $ is  a vector subspace (though not necessarily  closed) of  $ L^{2}(X, \ud\nu)$. 
Consider now the operators $G_S = C_{\Psi} \,S \,C_{\Psi}^{-1} : R_{C} \to C_{\Psi}(R_{S})$ and
$G_{S}^{-1}:=C_{\Psi} \,S^{-1} \,C_{\Psi}^{-1} : C_{\Psi}(R_{S}) \to R_{C}$, both acting in the Hilbert space $ \ov{R_C}$.
Then one shows \cite{jpa-sqintegI} that $G_{S} $ is a bounded, positive and symmetric operator, while $G_{S} ^{-1}$ is positive and essentially self-adjoint. 
These two operators are bijective and inverse to each other 
(detailed proofs  will be given for the discrete case in Section \ref{subsec:unbdd-discrframe} below).

{Thus the previous diagram \eqref{diagram1} becomes}
\be\label{diagram2}
\hspace*{-1cm}\begin{minipage}{15cm}
\begin{center}
\setlength{\unitlength}{2pt}
\begin{picture}(100,60)
\put(10,50){$\Hil$}
\put(89,50){${\h}_{\Psi}= R_{C} \subseteq  \ov{R_C}\subseteq L^{2}(X, \ud\nu) $}
\put(20,52){\vector(4,0){65}}
\put(55,54){$C_{\Psi}$}
\put(10,47){\vector(0,-4){30}}
Ê\put(12,17){\vector(0,4){30}}
\put(100,47){\vector(0,-4){30}}
Ê\put(102,17){\vector(0,4){30}}
\put(-35,10){$ \Hil \supseteq \dom(S^{-1})=R_S$}
\put(85,47){\vector(-2,-1){65}}
\put(14,30){$S^{-1}$}
\put(5,30){$S$}
\put(91,30){$G_{S}$}
\put(104,30){$G_{S} ^{-1}$}
\put(93,10){$C_{\Psi}(R_{S}) \subseteq L^{2}(X, \ud\nu)$}
\put(20,12){\vector(4,0){65}}
\put(48,35){$C_{\Psi}^\ast$}
\put(50,5){$C_{\Psi}$}
\end{picture}
\end{center}
\end{minipage}
\end{equation}

Next let   $G = \ov{G_{S}}$ and let $G^{-1}$ be the self-adjoint extension of  $G_{S} ^{-1}$. Both operators are self-adjoint and positive, 
$G$ is bounded and  $G^{-1}$ is densely defined   in $\ov{R_C}$. Furthermore, they are are inverse of each other on the appropriate domains.
 Moreover, since the spectrum of  $G^{-1}$ is bounded away from zero,
the norm $\norm{\Psi}{\cdot}$ is equivalent to the graph norm of {$G^{-1/2} = {\left( G^{-1} \right)}^{1/2}$}, so that
$$
 \ran(G^{1/2}) = \dom(G^{-1/2}) = {\h}_{\Psi} = R_{C} \subset \ov{R_{C}}\subset  L^{2}(X, \ud\nu).
$$ 
Therefore in this case $G^{- 1/2} = {\left( G^{-1} \right)}^{1/2} = {\left( G^{1/2} \right)}^{-1}$.

Since $C_{\Psi}^{-1} : {\h}_{\Psi} \to \Hil$ is unitary,  it is the adjoint of $C_{\Psi}: \Hil\to {\h}_{\Psi}$, denoted $C_{\Psi}^{\ast (\Psi)}$.
 Thus $G_{|_{\h_\Psi}} = C_\Psi S C_{\Psi}^{\ast (\Psi)}$ is a unitary operator, hence
 $ G $ and $G^{-1}$ are unitary images of $S$ and  $S^{-1}$, respectively, thus  $\|G \|_{\Psi}  = \|S\|_\Hil$.

At this point, we   make a distinction. We will say that the upper semi-frame  $\Psi=\{\psi_{x},\, x\in X\}$ is \emph{regular}
 if  all the vectors $\psi_{x}, \, x \in X$, belong to $\dom(S^{-1})$. This will simplify  some  statements below.

Indeed, let us first assume that $\Psi$ is regular. Then the discussion proceeds exactly as in the bounded case. In particular,  
the reproducing kernel $K(x,y) = \ip{\psi_{x} }{S^{-1}\psi_{y}}$ is a \emph{bona fide} function on $X\times X$.
 If  $\Psi$ is not regular, one has to treat the kernel $K(x,y)$ as a bounded sesquilinear  form
 over $ {\h}_{\Psi}$, i.e.,  use the language of distributions (see Section \ref{subsec:G-triplet}).

Under the same condition  of regularity, we obtain the same reconstruction formula as before,  but restricted to the subspace $R_S$,
the integral being understood in the weak sense, as usual:
\be\label{eq:reconstr2}
f = C_{\Psi}^{-1} F = C_{\Psi}^{\ast (\Psi)} F =  \int_X  F(x) \,S^{-1}\,\psi_{x} \; \ud\nu(x),  \;
  \forall\;f\in R_S,\;  F= C_{\Psi} f \in {\h}_{\Psi}.
\end{equation}
The argument goes as follows. 
 
Given $f' \in  R_S = \dom(S^{-1})$, we have, for  any  $F\in \h_{\Psi}\subseteq L^{2}(X, \ud\nu)$,
\begin{align} \label{eq:reconstrcont1}
\ip{C_{\Psi}^{\ast (\Psi)}F}{f'}_{\Hil} &= \ip{F}{C_{\Psi}f'}_\Psi \nn\\
&=\ip{F}{C_{\Psi}S^{-1}C_{\Psi}^{-1}C_{\Psi}f'}_{L^2} \nn\\
&=\ip{F}{C_{\Psi}S^{-1}f'}_{L^2}\nn \\
&=\int \ov{F(x)}\ip{\psi_x}{S^{-1}f'}_{\Hil} \ud\nu(x) 
\end{align}
This is true also for non-regular semi-frames. If we assume regularity and by using the fact that $S^{-1}$ is self-adjoint, we can write,
\begin{equation}\label{sec:regularreconweak1}
\ip{f}{f'}_{\Hil} = \ip{C_{\Psi}^{\ast (\Psi)}F}{f'}_{\Hil} = \int \ov{F(x)}\ip{S^{-1}\psi_x}{f'}_{\Hil} \ud\nu(x),
 \forall\;f\in R_S,\;  F= C_{\Psi} f \in {\h}_{\Psi}.
\end{equation}
 which then yields  \eqref{eq:reconstr2}.
However, it does not seem possible to extend this reconstruction formula to all $f\in\Hil$. A proof is given, in the discrete case, after Proposition \ref{ssminus}.

On the other hand, if the frame is not regular, we  have to turn to distributions, for instance, in terms of a Gel'fand triplet, as we show in the next section.

\subsection{Formulation in terms of a Gel'fand triplet}
\label{subsec:G-triplet}
 
The last remark becomes clearer if we formulate the whole construction in the language of Rigged Hilbert spaces or Gel'fand triplets
\cite{gelf}. Actually, we get here a simpler version, namely a triplet of Hilbert spaces, the simplest form of (nontrivial) partial inner product space 
\cite{jpa-pipspaces}.

According to    \cite[Sec.4]{jpa-contframes} and \cite[Sec.7.3]{jpa-CSbook}, the construction goes as follows.
If  $\Psi$ is regular, the kernel
\be\label{eq:kernel}
K(x,y) = \ip{\psi_{x} }{S^{-1}\psi_{y}}_{\Hil},
\end{equation}
is a \emph{bona fide} function on $X\times X$.
Indeed,  \eqref{eq:kernel},  together with \eqref{eq:reconstrcont1}, yields
\be\label{eq:sesqform1}
\iint_{X\times X} \ov{F(x)}K(x,y)F'(y)\; \ud\nu(x) \; \ud\nu(y) = \ip{C_{\Psi}^{-1} F}{SC_{\Psi}^{-1}F'}_\Hil, \; \forall\, F,F'\in {\h}_{\Psi}.
\end{equation}
Since $C_{\Psi}$ is an isometry and $S$ is bounded, the relation \eqref{eq:sesqform1} defines a bounded sesquilinear form over ${\h}_{\Psi}$, namely
\be\label{eq:sesqform2}
K^\Psi(F,F') = \ip{C_{\Psi}^{-1} F}{SC_{\Psi}^{-1} F'}_\Hil,
\end{equation}
and this remains  true even  if  $\Psi$ is not regular.  
Denote by ${\h}_{\Psi}^{\times}$ the Hilbert space obtained by completing  ${\h}_{\Psi}$ in the norm given by this sesquilinear form.
Now, \eqref{eq:sesqform1} and \eqref{eq:sesqform2}  imply that 
$$ 
 K^\Psi(F,F') = \ip{C_{\Psi}^{-1} F}{SC_{\Psi}^{-1} F'}_\Hil = \ip{F}{C_{\Psi} SC_{\Psi}^{-1} F'}_\Psi = \ip{F}{F'}_{L^2} .
$$
Therefore, one obtains, with continuous and dense range embeddings,
\be
 {\h}_{\Psi} \;\subset\; {\h}_{0}\;\subset \; {\h}_{\Psi}^{\times},
\label{eq:triplet}
\end{equation}
where 
\begin{itemize}
\vspace{-1mm}\item[{\bf .}]  $ {\h}_{\Psi} = R_{C}$, which is a Hilbert space for the norm 
$\norm{\Psi}{\cdot}=\ip {\cdot}{ G^{-1}  \cdot }^{1/2}_{L^2}$; 
\vspace{-1mm}\item[{\bf .}]   ${\h}_{0} =\ov{{\h}_{\Psi}} = \ov{R_{C}}$ is the closure of ${\h}_{\Psi}$ in $L^{2}(X, \ud\nu)$;
\vspace{-1mm}\item[{\bf .}]  
${\h}_{\Psi}^{\times}$ is 
the completion of  ${\h}_{0}$  (or ${\h}_{\Psi}$) in the norm $\norm{\Psi^{\times}}{\cdot}:=\ip {\cdot}{G \cdot }^{1/2}_{L^2}$.
\end{itemize}
 The notation in \eqref{eq:triplet} is coherent, since  the space $\h^{\times}_{\Psi}$ just constructed is the conjugate dual of ${\h}_{\Psi}$, 
 i.e. the space of conjugate linear functionals on   ${\h}_{\Psi}$ 
(we use the conjugate dual instead of the dual, in order that all embeddings in \eqref{eq:triplet} be linear). 
Indeed, since $K^\Psi$ is a bounded sesquilinear form over  ${{\h}_{\Psi}}$, the relation $ X_F= \ov {K^\Psi(F,\cdot)}$
 defines, for each $F\in{\h}_{\Psi}$,  an element $X_F$ of the conjugate dual of  ${\h}_{\Psi}$ (note that  $X_F$ depends linearly on $F$). 
If, on these elements, we define the inner product
 $$
 \ip{X_F}{X_{F'}}_{{\Psi}^{\times}} = \ip{C_{\Psi}^{-1} F}{SC_{\Psi}^{-1} F'}_\Hil = K^\Psi(F,F')
 $$
 and take the completion, we obtain precisely the Hilbert space ${\h}_{\Psi}^{\times}$. Thus \eqref{eq:triplet} is a Rigged Hilbert space or a Gel'fand triplet, 
more precisely a Banach (or Hilbert) Gel'fand triple in the terminology of Feichtinger \cite{fei-zimm}.
 
 Now, given any element $X\in{\h}_{\Psi}^{\times}$, we easily obtain by a limiting procedure that
 $$
 X(F) = \ip{X_F}{X}_{{\Psi}^{\times}} = \ip{\ov {K^\Psi(F,\cdot)}}{X}_{{\Psi}^{\times}},
 $$
 which expresses the reproducing property of the kernel $K^\Psi$ as a function over $\ov{{\h}_{\Psi}} \times {\h}_{\Psi}$.
Clearly,  ${\h}_{\Psi}^{\times}$  {materializes the unbounded character of the upper semi-frame.}

 Let us come back now to the relation \eqref{eq:reconstrcont1}, that is,
$$
 \ip{f}{f'}_{\Hil}=\ip{F}{C_{\Psi}S^{-1}f'}_{L^2}, \; f\in \Hil, \, f'\in R_{S}.
$$
Define $F':=C_{\Psi}S^{-1}f' \in {\h}_{\Psi}$. If $\Psi$ is not regular, we can associate to  $F'$   an element  $X_{F'}\in {\h}_{\Psi}^{\times}$, namely,
$$
X_{F'}(F) = K^\Psi(F,F') = \ip{F}{F'}_{L^2} = \ip{f}{f'}_{\Hil}.
$$
Although these equations give some way of inverting the analysis operator, they don't give explicit reconstruction formulas.

If $S^{-1}$ is bounded, that is, in the case of a frame, the three Hilbert spaces of \eqref{eq:triplet} coincide as sets, with equivalent norms,
 since then both $S$ and $S^{-1}$  belong to $GL(\Hil)$.

Finally, if $\Psi$ is regular, all three spaces $ {\h}_{\Psi},{\h}_{0}, {\h}_{\Psi}^{\times}$ are  {reproducing kernel Hilbert spaces}, with the same kernel
  $K(x,y) = \ip{\psi_{x} }{S^{-1}\psi_{y}}$. 

\subsubsection{An alternative}

A possibility to have a more general reconstruction formula is the following. In the triplet \eqref{eq:triplet}, the `small' space ${\h}_{\Psi}$ is the form domain of 
$ G^{-1}$, with norm $\norm{\Psi}{\cdot}=\ip {\cdot}{ G^{-1}  \cdot }^{1/2}_{L^2}$.  On the side of $\Hil$, i.e., on the l.h.s. of the diagram \eqref{diagram2},
this corresponds to the form domain of $S^{-1} $, with norm $\norm{\Psi}{\cdot}'=\ip {\cdot}{S^{-1}  \cdot }^{1/2}_{\Hil}$. We can instead consider the smaller space 
  $R_S = \dom(S^{-1}) $, with norm $\norm{\s}{\cdot}=\ip {S^{-1} \cdot}{S^{-1}  \cdot }^{1/2}_{\Hil}$. The resulting space, denoted $\s$, is complete, hence a Hilbert space.
Thus we get a new triplet
\be\label{eq:newtriplet}
 \s \subset \Hil \subset  \s^{\times}.
\end{equation}
Mapping everything into $L^2$ by $C_{\Psi}$, we obtain the following scale of  Hilbert spaces:
\be
\h_{\s} \subset \h_{\Psi} \subset \h_{0} = \ov{R_C} \subset {\h}_{\Psi}^{\times} \subset {\h}_{\s}^{\times}.
\end{equation}
In this relation,   $\h_{\s}:= C_{\Psi} \s=  C(R_{S})$, with norm $\norm{\Hil}{G^{-1}f}$. In the new triplet \eqref{eq:newtriplet}, the operator $S^{-1}$ is isometric from
$ \s$ onto $\Hil $ and, by duality, from $\Hil $ onto $ \s^{\times}$ .
The benefit of that construction is that the relation \eqref{sec:regularreconweak1} is now valid for any  $f\in \Hil$, even if $\Psi$ is non-regular, but of course, we still need 
that $f'\in \s = R_{S}$. In other words, we obtain  a reconstruction formula in the sense of distributions, namely,
\be\label{sec:regularreconweak2}
f = \int \ov{F(x)}\,S^{-1}\psi_x \ud\nu(x), \; f \in \s^{\times}.
\end{equation}

\subsection{Upper and lower semi-frames}
\label{subsec:upperlower}

Given a frame $\Psi = \{\psi_{x}\}$, with frame bounds ({\sf m}, {\sf M}) and frame operator $S$,
it is well-known that the family  $\widetilde\Psi = \{\widetilde\psi_{x}:= S^{-1}\psi_{x}\}$ is also a frame, with bounds $({\sf M}^{-1}, {\sf m}^{-1}) $
and frame operator $S^{-1}$, called the
\emph{canonical dual} of $\Psi$. It follows that the reconstruction formula \eqref{eq:reconstr} may be written as (the integrals being understood in the weak sense, as usual)
\be
f = \int_X  \ip{\psi_{x}}{f} \,\widetilde\psi_{x} \; \ud\nu(x) = \int_X  \ip{\widetilde\psi_{x}}{f} \,\psi_{x} \; \ud\nu(x) . 
\label{eq:dualreconstr}
 \end{equation}
More generally, one says \cite{zakharova} that a frame $\{\chi_{x}\}$ is \emph{dual} to the frame $\{\psi_{x}\}$ if one has, for every $f\in\Hil$,
\be
f = \int_X  \ip{\chi_{x}}{f} \,\psi_{x} \; \ud\nu(x).
\label{eq:dual}
 \end{equation}
It follows that the frame $\{\psi_{x}\}$ is dual to the frame  $\{\chi_{x}\}$. We want to extend this notion to semi-frames.

First we notice \cite{gabardo-han} that an  upper semi-frame  $\Psi$ is a frame if and only if there exists another  upper semi-frame  $\Phi$ which is dual to $\Psi$, in the sense that
$$
\ip{f}{f'} = \int_X  \ip{\phi_{x}}{f} \,\ip{\psi_{x}}{f'} \; \ud\nu(x), \; \forall\, f,f' \in \Hil. 
$$

Now, an upper semi-frame  $\Psi$ corresponds formally to ${\sf m} \to 0$ in \eqref{eq:frame}, thus, it yields $S$ bounded,  $S^{-1}$ unbounded. 
Thus the `dual' $\widetilde\Psi$  should be a family satisfying
 the lower frame condition (no finite upper bound, $M \rightarrow \infty$), which would then  correspond to $S$ unbounded and  $S^{-1}$ bounded. 
As we will see in the sequel, this idea is basically correct, with some minor qualifications.
Indeed, there is perfect symmetry (or duality) between two classes of {total}  families, namely,  the upper  semi-frames  and the lower  semi-frames.

Let first $\Psi = \{\psi_{x}\}$ be an arbitrary total family in $\Hil$.
Then we define the analysis operator $C_{\Psi}: \dom(C_{\Psi})\to L^{2}(X, \ud\nu)$ as $C_{\Psi}f(x) = \ip{\psi_{x}}{f}$ on the domain
$$
\dom(C_{\Psi}):= \{f\in\Hil : \int_{X}  |\ip{\psi_{x}}{f}| ^2 \, \ud \nu(x) <\infty\}.
$$
In parallel to the discrete case, \cite[Lemma 3.1]{casoleli1}, we can state:
\belem
Given any total family $\Psi =\{\psi_{x}\}$, the associated analysis operator $C_{\Psi}$ is closed. Then $\Psi $ satisfies the lower frame condition
 if and only if $C_{\Psi}$ has closed range and is injective.
\label{lem:analop}
\enlem
\begin{proof}
To show that $C_{\Psi}$ is closed, we   have to show that, if $f_n \rightarrow f$ and $C_{\Psi} f_n \rightarrow g$, then  $f\in\dom(C_{\Psi})$ and $C_{\Psi} f = g$.
If $f_n \rightarrow f$, then this sequence is also weakly convergent.
 In particular, for $n\to\infty$ and almost all  $x $, we have  $C_{\Psi} f_n(x) = \ip{\psi_x} {f_n} \to \ip {\psi_x}{f}=C_{\Psi} f(x) $.
 As by assumption $g \in L^2$, this implies that   $ \ip {\psi_x}{f}  \in L^2$, that is, $f\in\dom(C_{\Phi})$ and  $C_{\Psi} f = g $.  

Next, the existence of the lower frame bound implies that $C_{\Psi}$ is injective, hence invertible.
Since $C_{\Psi}$ is closed, $C_{\Psi}^{-1}$ is closed. Thus, by the closed graph theorem \cite{conw1} or \cite[Theor. 5.6]{weidmann}, 
$C_{\Psi}$ has closed range if and only if 
 $C_{\Psi}^{-1}$ is continuous on $\ran(C_{\Psi})$, which is  equivalent to the existence of a lower frame bound.
\end{proof}
\medskip

Next,    we define the synthesis operator $D_{\Psi}: \dom(D_{\Psi})\to\Hil $ as 
\be
D_{\Psi}F =  \int_X  F(x) \,\psi_{x} \; \ud\nu(x),  \;\;F\in \dom(D_{\Psi}) \subset L^{2}(X, \ud\nu) ,
\label{eq:synthmap2}
\end{equation}
on the domain 
$$
\dom(D_{\Psi})  := \{F\in L^{2}(X, \ud\nu)  : \int_{X} F(x) \,\psi_{x}  \, \ud \nu(x)  \mbox{ converges weakly in $\Hil$ }\}.  
$$
Then we have, as in the discrete case \cite[Lemma 3.1 and Prop. 3.3]{BSA}:
\belem\label{lem:locinteg}
If the   total family $\Psi =\{\psi_{x}\}$ is  such that the function $x\mapsto \ip{\psi_{x}}{f}$ is locally integrable for all $f\in\Hil$, then the operator 
$D_{\Psi}$ is densely defined and one has $C_{\Psi} = D_{\Psi}^\ast$.
\enlem
\begin{proof}
Since $x\mapsto \ip{\psi_{x}}{f}$ is locally integrable, 
 the domain $\dom(D_{\Psi})$ contains the characteristic functions of all compact subsets of $X$, which are dense in $L^{2}(X, \ud\nu)$,
 thus  $D_{\Psi}$  is densely defined.

Next, for any $F\in\dom(D_{\Psi})$ and $f\in \dom(C_{\Psi})$, we have
$$ 
\ip{D_{\Psi}F}{f}_{\Hil}  = \int_{X}\ov{F(x)}\,\ip{\psi_{x}}{f}\, \ud \nu(x) = \ip{F}{C_{\Psi}f}_{L^2},
$$
which implies that  $C_{\Psi} \subseteq D_{\Psi}^\ast$.

It remains to show that $ \dom(D_{\Psi}^\ast) \subseteq \dom(C_{\Psi}) $.
Fix $f \in \dom(D_{\Psi}^*)$. This means that ${\mathfrak d} (F) := \ip {f}{D_{\Psi} F} $, is a bounded functional on $\dom(D_{\Psi})$.  
Since $\dom(D_{\Psi})$ is dense in $L^2$, there is a unique bounded extension $\ov{\mathfrak d} : L^2 \rightarrow \CC$. 
For $F\in L^2$, we denote by $F_{n}:=F \, \chi_{K_{n}} $ its restriction to $K_{n}$, where  $\chi_{K_{n}}$ is the characteristic function of $K_{n}$ and
$X = \bigcup_{n} K_{n}, K_{n} \subset  K_{n+1}, \,K_{n}$ relatively compact. 
Thus, $F_{n}  \in\dom(D_{\Psi})$.
Clearly, $F_{n} \to F$ in $L^2$-norm as $n\to\infty$. Hence, $\ov{\mathfrak d}  (F_{n}) \to\ov{\mathfrak d} (F)$
 as $n\to\infty$. Therefore,  
$$
\ov{\mathfrak d} (F_{n})= \int_{K_{n}} \ov{F(x)} \,\ip{\psi_k}{f} \, \ud \nu(x)  \stackrel{n\to\infty}{\longrightarrow}
\int_{X} \ov{F(x)} \,\ip{\psi_k}{f}, \; \mbox{ for all }F\in L^2(X, \ud\nu). 
$$
Since $L^2$  is its own K\"{o}the dual (see \cite{lux} or \cite[Sec.4.4]{jpa-pipspaces}), this implies that $ \ip{\psi_x}{f}   \in  L^2(X, \ud\nu)$, 
which proves that $f$ belongs to $\dom(C_{\Psi})$. 
\end{proof}
\medskip

Note that the condition of local integrability is certainly satisfied for every $f\in\dom(C_{\Psi})$, but not necessarily for every $f\in\Hil$. It is always satisfied for
an upper semi-frame, since then $\dom(C_{\Psi})=\Hil$.
\medskip

Finally, one defines the frame operator as $S=D_{\Psi}  C_{\Psi}$, so that
$$
 {Sf} =  \int_{X}  \ip{\psi_{x}}{f} \psi_{x} \, \ud \nu(x) , \;\; \forall \,f\in \dom(S),  
$$
where
$$
\dom(S)  := \{f\in \Hil  : \int_{X} \ip{\psi_{x}}{f}\, \psi_{x}   \, \ud \nu(x)  \mbox{ converges weakly in $\Hil$ }\}.
$$
Notice that one has in general $\dom(S) \subsetneqq \dom(C_{\Psi})$. As in the discrete case  \cite[Lemma 3.1]{BSA}, one has
$\dom(S) = \dom(C_{\Psi})$ if and only if $\ran(C_{\Psi})\subseteq \dom(D_{\Psi}) $. This happens, in particular, for an upper semi-frame $\Psi$, for which one has
$\dom(S) = \dom(C_{\Psi}) = \Hil$.
\medskip

For an upper  semi-frame,   $S:\Hil\to\Hil$ is a bounded injective operator and $S^{-1}$ is unbounded. If $\Phi=\{\phi_{x}\} $ satisfies the lower frame condition,  
 with lower frame bound  {\sf m}, then $S: \dom(S)\to\Hil$ is an injective operator, possibly unbounded, with a bounded inverse $S^{-1}$. Indeed,  
 if $Sf=0$ for some $f\in \dom(S)$, then $\ip{f}{Sf}  = \int_{X}  |\ip{\phi_{x}}{f}| ^2 \, \ud \nu(x)=0$, which implies that $f=0$, because $\Phi$ is total.
Furthermore, $S$ is bounded from below, since, for $f\in \dom(S)$, one has ${\sf m}\|f\|^2\leq \ip {f}{Sf}\leq \|Sf\|\cdot\|f\|$, which implies that
 $\|S^{-1}g\|\leq  {\sf m} ^{-1}\|g\|$, $\forall \,g\in \dom(S^{-1})$, i.e., $S^{-1}$ is bounded. 
\medskip

Actually,  the lower frame condition by itself is not sufficient to eliminate a number of pathologies. For instance, $S$ and $C_{\Psi}$ could have nondense domains, even reduced to
$\{0\}$, in which case one cannot define a unique adjoint $C_{\Psi}^*$ and $S$ may not be self-adjoint. One way to avoid these bad situations is the following (in the discrete case, 
a similar statement is given  in \cite[Prop 4.5]{ole2}  and \cite[Lemma 3.1(iv)]{BSA}).
\belem\label{lem:densedom}
Let $\psi_{y}\in \dom(C_{\Psi}),\, \forall\, y\in X$. Then $C_{\Psi}$ is densely defined.
\enlem
\begin{proof}
 First observe that $\overline{{\rm span}(\Psi)}^{\bot}\subseteq \dom(C_{\Psi})$. Assume that  $\psi_{y}\in \dom(C_{\Psi}),\, \forall\, y\in X$. 
  Then ${\rm span}(\Psi)\subseteq \dom(C_{\Psi})$. Therefore, $\Hil=\overline{{\rm span}(\Psi)} \oplus \overline{{\rm span}(\Psi)}^{\bot} \subseteq \overline{\dom(C_{\Psi})}$,
 which implies that $\dom(C_{\Psi})$ is dense in $\Hil$.
 \end{proof}
\medskip

Under the condition of Lemma \ref{lem:densedom}, there is a unique adjoint $C_{\Psi}^*$ and one has $D_{\Psi}\subseteq C_{\Psi}^*$, so that $D_{\Psi}$
is closable. One has indeed, as in the proof of Lemma \ref{lem:locinteg},  
$$
\ip{D_{\Psi}F}{f}_{\Hil}  = \int_{X}\ov{F(x)}\,\ip{\psi_{x}}{f}\, \ud \nu(x) = \ip{F}{C_{\Psi}f}_{L^2}, \quad
 \mbox{for any $F\in\dom(D_{\Psi})$ and $f\in \dom(C_{\Psi})$}.
$$
Thus we may state:
\belem\label{lem:densedom2}
(i) If $C_{\Psi}$ is densely defined, then $D_{\Psi}\subseteq C_{\Psi}^*$ and $D_{\Psi}$ is closable.  

(ii) $D_{\Psi}$ is closed if and only if $D_{\Psi}= C_{\Psi}^*$. In that case, $S=C_{\Psi}^*C_{\Psi} $ is self-adjoint.
\enlem
\begin{proof}
The `if' part  is obvious. If $D_{\Psi}$ is closed, then  $D_{\Psi}^*$ is densely defined and   $D_{\Psi}^{**} = D_{\Psi}$, which is  $C_{\Psi}^*$, since
$C_{\Psi} = D_{\Psi}^\ast$.
\end{proof}
\medskip

In view of the duality results of Section \ref{subsec:duality} below, we say that  a   
family $\Phi=\{\phi_{x}\}$
 is a \emph{lower semi-frame} if 
it satisfies the lower frame condition, that is, there exists   a constant ${\sf m}>0$ such that
\be
{\sf m}  \norm{}{f}^2 \le  \int_{X}  |\ip{\phi_{x}}{f}| ^2 \, \ud \nu(x) , \;\; \forall \, f \in \Hil.
\label{eq:lowersf}
\end{equation}
Clearly, \eqref{eq:lowersf} implies that the family $\Phi $ is total in $\Hil$.

On the other hand, a lower semi-frame is not a coherent state, as we have defined them in Section \ref{sec:intro}. 
Indeed,  the latter requires $S$ to be bounded, whereas here it could be unbounded, not densely defined and even have domain reduced to $\{0\}$. 
Actually, we could define such a frame operator $S$ for any family of vectors, at the risk of getting such
pathologies, but the same is true in the discrete case also, as shown in Section 4 of \cite{ole2} or \cite{BSA}.

Alternatively, one could define a lower semi-frame in a more restricted way, as a   total family $\Phi$ for which
$S$ is densely defined, bijective onto $\Hil$ and closed. Then it follows that $C_{\Phi}$ is also densely defined and that $S^{-1}$ is bounded,
 since it is closed and everywhere defined. However, this class seems to be too narrow, it is not the natural dual class of upper semi-frames.

\subsection{Duality between upper and lower semi-frames}
\label{subsec:duality}

Now we turn towards duality between upper and lower semi-frames. The first result is immediate.
\belem
Let $\Psi = \{\psi_{x}\}$ be an upper semi-frame, with upper frame bound {\sf M} and let $\Phi =\{\phi_{x}\}$ be a total family dual to $\Psi$.
Then $\Phi$ is a lower  semi-frame, with lower frame bound ${\sf M}^{-1}$. 
\label{lem:dual1}
\enlem
\begin{proof}
Duality means that
$$
f = \int_X  \ip{\phi_{x}}{f} \, \psi_{x} \, \ud\nu(x)  , \;\; \forall \,f\in \Hil.
$$
Then 
\begin{align*}
\norm{}{f}^2 &= \ip{f}{f} = \int_X  \ip{\psi_{x}}{f} \, \ip{f} {\phi_{x}}\, \ud\nu(x)
\\
&\leq \Big( \int_X |\ip{\psi_{x}}{f}| ^2  \, \ud\nu(x)  \Big)^{1/2}\, \Big( \int_X |\ip{\phi_{x}}{f}| ^2   \; \ud\nu(x) \Big)^{1/2}
\\
&\leq {\sf M}^{1/2}\norm{}{f}  \Big(\int_X  |\ip{\phi_{x}}{f}| ^2  \, \ud\nu(x)  \Big)^{1/2},
\end{align*}
so that 
$$
 {\sf M}^{-1}\norm{}{f} ^2 \leq \int_X  |\ip{\phi_{x}}{f}| ^2  \, \ud\nu(x) .
$$
\end{proof}

Note that, according to the remark above, it might happen that the lower semi-frame $\Phi$ is in fact a frame.

However, there is a  stronger result, namely (for the discrete case  see \cite[Prop.3.4]{casoleli1}), 
\begin{prop.}\label{prop:duality}  
Let $\Phi =\{\phi_{x}\}$ be any total family in $\Hil$. Then $\Phi $ is a lower semi-frame if and only if there exists an upper semi-frame
$\Psi $ dual to $\Phi $,  that is, one has, in the weak sense,
$$
f = \int_X  \ip{\phi_{x}}{f} \, \psi_{x} \, \ud\nu(x)  , \;\; \forall\, f\in \dom(C_{\Phi}).
$$
\end{prop.}
\begin{proof}
The `if' part is Lemma \ref{lem:dual1}. Let $\Phi $ be a lower semi-frame. 
Then  $C_{\Phi}^{-1}: \ran(C_{\Phi})\to \Hil$ is bounded. 
Define a linear operator $V: L^{2}(X, \ud\nu)\to \Hil$ by $V=C_{\Phi}^{-1}$ on  $\ran(C_{\Phi})$, by $V=0$ on $\ran(C_{\Phi})^{\bot}$
and extending by linearity. Then $V$ is bounded. 

Let now $\{F_{j}, j=1,\ldots,\infty\}$ be an arbitrary orthonormal basis of  $ L^{2}(X, \ud\nu)$. Writing $w(x):=C_{\Phi}f(x)= \ip{\phi_{x}}{f}$ 
for  $f\in \dom(C_{\Phi})$, we have $w= \sum_{j}  \ip{F_{j}}{w} F_{j}$. Then,
\begin{align*}
f = VC_{\Phi}f &= \sum_{j}  \ip{F_{j}}{w}V F_{j}
\\
& =\sum_{j} \Big( \int_X \ov{F_{j}(x)}w(x)\, \ud\nu(x)\Big)V F_{j}
\\
& =\int_X w(x) \Big( \sum_{j} \ov{F_{j}(x)} V F_{j}\Big)\, \ud\nu(x) 
\\
& =\int_X \ip{\phi_{x}}{f}\,\psi_{x}\, \ud\nu(x),
\end{align*}
where we have defined $\psi_{x}:=  \sum_{j} \ov{F_{j}(x)} V F_{j}$, with the sum converging weakly.

The  interchange between sum and integral may be justified by a limiting argument.
Then, using the orthonormality of the basis $\{F_{j}\}$, we get, for every $g\in \Hil$,
\begin{align*}
\int_X |{ \ip{\psi_{x}}{g}|^2\, \ud\nu(x)}
&= \int_X | \sum_{j} \ov{F_{j}(x)} \ip{V F_{j}}{g} |^2\, \ud\nu(x)\\
&= \int_X \sum_{j,k} \ov{F_{j}(x)}F_{k}(x) \ip{V F_{j}}{g}\ip{g}{V F_{k}\, \ud\nu(x)}\\
&= \sum_{j}|\ip{V F_{j}}{g} |^2 
 \leq c \norm {}{g}^2,
\end{align*}
since $(V F_{j})$ is a total Bessel sequence \cite{BSA}. Thus $\Psi=\{\psi_{x}\}$ is indeed an upper semi-frame, dual to $\Phi =\{\phi_{x}\}$. 
\end{proof}

\subsection{An example of a non-regular upper semi-frame}
\label{subsec:example}

We shall illustrate the situation by briefly describing an example of an upper semi-frame that is not regular, namely, the affine coherent states introduced by Th.~Paul \cite{paul}
(see also \cite{jpa-contframes}). These coherent states stem from the unitary irreducible representation of the connected affine group $G_{A}$ or $ax+b$ group
$$
(U_{n}(a,b)f )(r) = a^{\frac n2}e^{-ibr}f(ar), \; a>0, \, b \in\RN, \; f \in \Hil_{n},
$$
where $\Hil_{n}:=L^{2}({\RN}^{+}, r^{n+1}\ud r), n = \hbox{integer} \geq 1 $. The coherent states are indexed by points of the quotient space $G_{A}/H\simeq \RN$,
where $H$ denotes the subgroup of dilations. Since the representation $U_{n}$ is square integrable mod$(H,\sigma)$, for a suitable section $\sigma : \RN\to G_{A}$,
coherent states may be constructed by the general formalism \cite{jpa-sqintegI,jpa-CSbook}. They take the form
$$
\psi_{x}(r) = e^{-ixr}\psi (r), \quad r \in {\RN}^{+},
$$
where  $\psi$ is admissible if it satisfies  the two conditions
 \begin{align} 
  (i) \;  & \sup_{r \in {\RN}^{+}}[2\pi r^{n-1}|\psi (r)|^{2}] = 1 , \label{eq518}
\\
 (ii) \; & |\psi (r)|^{2} \neq 0, \; \hbox{except perhaps  at isolated points} \; r \in {\RN}^{+}\nn.   
\end{align}
The frame operator $S$ is   a multiplication operator on $\Hil_{n}$,
$$ 
(Sf )(r) = 2\pi r^{n-1}|\psi (r)|^{2}f(r).     
 $$
Thus 
$$
(S^{-1}f)(r) = \frac 1{2\pi}\frac {r^{1-n}}{|\psi (r)|^{2}}f (r),
$$
which  is unbounded. In addition, $f\in \dom (S^{-1})$ if and only if 
\be\label{eq522}
\frac 1{(2\pi )^{2}}\int_{{\RN}^{+}} \frac {r^{2-2n}}{|\psi (r)|^{4}}   |f(r)|^{2}r^{n+1}\ud r  < \infty . 
\end{equation}
Thus, comparing the conditions \eqref{eq518} and \eqref{eq522}, we see that none of the vectors $\psi_{x}$ is
in the domain of either $S^{-1/2}$ or $S^{-1}$. The coherent state map
$   C_{\Psi}: \Hil_{n} \to  L^{2}({\RN} , \ud x)$, given by
\be \label{eq:analexamp1} 
 (C_{\Psi}f )(x)  = \ip{ \psi_{x}}{f}     = \int_{{\RN}^{+}} e^{ixr}\ov{\psi (r)} f(r)r^{n+1}\ud r  ,     
\end{equation}
is   bounded and its range in $L^{2}({\RN} ,\ud x)$ is closed
in the new norm
\be\label{eq:normpsi}
\norm{\Psi} {F} ^{2} =  \ip {F}{C_{\Psi} \,S^{-1} \,C_{\Psi}^{-1} F }_{L^{2}({\RN}, \ud x)}
 =  \ip {F}  {G^{-1} \,  F }_{L^{2}({\RN}, \ud x)}.
\end{equation}
Finally, the reproducing kernel 
$$
 K(x,y) = ``\ip{\psi_{x}} {S^{-1} \psi_{y}}"   = \frac 1{2\pi}\int_{{\RN}^{+}}e^{i(x-y)r}\ud r,
 $$
is a distribution which defines a bounded sesquilinear form on $\h_{\Psi}$.

In this case, everything can be computed explicitly (in the sequel, we will freely exchange integrals, which can be justified by a limiting procedure). 
First, following \eqref{eq:reconstrcont1}, we evaluate the function  
$$
{C_{\Psi}^{-1} F(r) =  \frac{\widehat F(r)}{ r^{n-1} \,\overline{\psi(r)}},} \; \mbox{where}
\;\widehat F(r):=\frac{1}{2\pi}\int_{-\infty}^{\infty} e^{-ixr}F(x)\ud x
\; \mbox{is the Fourier transform of}\; F.
$$

Next, using the expressions given above for $S$ and $S^{-1}$, we obtain   for $F \in C_\Psi \left( R_S\right)$:
\begin{align*}
G^{-1}F(x)& = C_{\Psi} S^{-1}C_{\Psi}^{-1}F(x) =  \frac{1}{2\pi}\int_{0}^{\infty} e^{ixr}\frac{\widehat F(r)}{|\psi(r)|^2 \,r^{n-3}}\; \ud r, \\
G F(x)& =  C_{\Psi} SC_{\Psi}^{-1}F(x) =  2\pi\int_{0}^{\infty} e^{ixr} {\widehat F(r)}\; | \psi(r)|^2\, r^{n+1}\, \ud r .
\end{align*}
From these relations, we get  the following norms for the three Hilbert spaces of \eqref{eq:triplet},  where we understand the formulas as integrals on the subsets 
  where they converge and extend them by closure to the whole set:
\bei
\vspace{-1mm}\item[{\bf .}]
For $\h_{\Psi} :   \norm{\Psi}{F}^2 = \ip{F}{G^{-1}F}_{L^2} = 
 \dis\int_{0}^{\infty}\frac{|\widehat F(r) |^2}{\left| \psi(r) \right|^2 \, r^{n-3}}\, \ud r $ ;

\item[{\bf .}]
For $\h_{0}  :  \norm{0}{F}^2 = \norm{L^2}{F}^2$ ;

\item[{\bf .}]
For $\h_{\Psi}^\times :   \norm{\Psi^\times}{F}^2 = \ip{F}{GF}_{L^2} 
= 4 \pi^2  \dis\int_{0}^{\infty} {\left|\widehat F(r) \right|^2} \;\left| \psi(r) \right|^2 r^{n+1} \ud r   $.
\eni

As a matter of fact, this example may be trivially discretized. It suffices to choose an infinite sequence of points $\{x_{k}, k= 1, 2, \ldots\}$.
Then the map $C_{\Psi}$ becomes $C:  \Hil_{n} \to  \ell^{2}$, namely 
$$
 (C f )_{k}  = \ip{ \psi_{x_{k}}}{f}     = \int_{{\RN}^{+}} e^{ix_{k}r}\ov{\psi (r)} f(r)r^{n+1}\ud r  ,  
 $$
 it is   bounded and its range in $\ell^{2} $ is closed in the new norm 
$$ 
 \norm{\Psi} {d} ^{2} =  \ip {d}{C \,S^{-1} \,C^{-1} d }_{\ell^{2} } =  \ip {d}  {G^{-1} \, d }_{\ell^{2} }. 
 $$
 Finally, the reproducing matrix becomes 
$$
  K_{kl} = ``\ip{\psi_{k}} {S^{-1} \psi_{l}}"   = \frac 1{2\pi}\int_{{\RN}^{+}}e^{i(x_{k}-y_{l})r}\ud r, 
 $$
  a well-defined distribution.

\subsection{A dual example: a lower semi-frame}
\label{subsec:dualexample}

We also provide an example of a lower semi-frame, namely,
an example where $S^{-1}$ is bounded, but $S$ is unbounded. In other words, the lower frame condition is satisfied, but the upper one is not.

The original construction of a continuous wavelet transform on the 2-sphere $\mathbb S^{2}$ \cite{jpa-wav2sph} was based 
on a square integrable representation of the Lorentz group. Starting from a (mother) wavelet $\psi\in \Hil =L^{2}(\mathbb S^{2}, \ud \mu)$,
 assumed to be axisymmetric for simplicity, one obtains
the whole family $\{ \psi_{a,\varrho}:=  R_\varrho \, D_a \psi ,\, (\varrho, a)\in X= \mbox{SO}(3)\times {\mathbb R}_+^*\}$, where $R_\varrho $, resp.
 $D_a,\, a>0$, is the unitary rotation, resp. dilation, operator in $L^{2}(\mathbb S^{2}, \ud \mu)$.

In that context,  the resulting frame operator $S$  is diagonal in Fourier space (harmonic analysis on the 2-sphere reduces to expansions in spherical harmonics
 $Y_l^m, \, l \in \NN, m= -l, \ldots, l$), thus it is a Fourier multiplier:
$$
\widehat{S f}(l,m) =  s(l)\widehat{f}(l,m) 
$$
where   
$$
s(l): = \frac{8\pi^2}{2l+1} \sum_{|m| \leq l}\int_{0}^\infty  |\widehat{\psi}_a(l,m)|^2  \,\frac{\ud a}{a^3}, \quad
 \forall \;l \in \NN.
$$
Here $\widehat{\psi}_a(l,m) = \ip {Y_l^m }  {\psi_a} $ is the Fourier coefficient of $\psi_a = D_a \psi$.

Then, the result of the analysis is twofold. First, the wavelet $\psi\in L^2(\mathbb S^{2},\ud\mu)$ is admissible (in the sense of group theory, that is, 
admissible with respect to a square integrable representation)
 if and only if there exists a constant  $ c>0$ such that
$s(l)  \leq c, \; \forall \;l \in \NN,$
equivalently, if the frame operator  $S$  is bounded and invertible
(actually the same condition was derived directly by Holschneider \cite{hol-sphere}). Then, for any axisymmetric wavelet $\psi$, there exists
 a constant $d>0$   such that $d \leq s(l)  \leq c, \; \forall \;l \in \NN$. Equivalently, $S$ and $S^{-1}$  are both bounded,
 i.e., the family of spherical wavelets 
 $\{ \psi_{a,\varrho} ,\, (\varrho, a)\in X= \mathrm{SO}(3)\times {\mathbb R}_+^*\}$  is a continuous frame. 
  One notices, however, that the upper frame bound, which is implied by the constant $c$, does depend on $\psi$, whereas the lower frame bound, 
which derives from $d$, does not, it follows from the asymptotic behavior of the function $Y_l^m$ for large $l$. 

 For any axisymmetric admissible  wavelet $\psi$, the wavelet transform of $f\in\Hil$ reads
 $$
W_f(\varrho, a) = \ip {\psi_{a,\varrho}}{f} = \int_{\mathbb S^{2}} \ov{[R_{\varrho}D_a \psi](\omega)}\,f(\omega) \, \ud\mu(\omega)
$$
and the corresponding reconstruction formula is
$$
f(\omega) = \int_{\RN_{+}^*}\int_{\mbox{\scriptsize  SO}(3)}  
W_f(\varrho, a)\,[S^{-1}R_{\varrho}D_a\psi](\omega)\, \frac{\ud a}{a^3}\, \ud\varrho   , \; f \in L^2(\mathbb S^{2}).
$$

Now, it was shown by Wiaux \emph{et al.} \cite{wiaux} that the same reconstruction formula is valid under the weaker condition
$0 < s(l) < \infty,\;  \forall \;l \in \NN$. Since the behavior of $s(l)$ is arbitrary, this means exactly that  the frame operator
$S$ is allowed to be unbounded. The lower frame bound, being independent of $\psi$, remains untouched, so that $S^{-1}$ stays bounded,
as announced.

\section{The discrete case}
\label{sec: discreteframes}

In the continuous case of Section \ref{sec:cont-frames}, the integrals are considered in the weak sense.
However, in the discrete case, we are interested in expansions with norm convergence, thus all the expansions in the rest of the paper should be understood as norm convergent. 

\subsection{Notations}
\label{subsec: notation}

Let now $X$ be a discrete set and  $\nu$ a counting measure, so that we go back to the familiar (discrete) frame  setting.
Before proceeding, it is useful to convert the notations.

The ambient Hilbert space  $L^{2}(X, \ud\nu )$ becomes $\ell^2 $. 
The sequence  $\Psi =(\psi_n, \,n \in \Gamma)$, where $\Gamma$ is some index set, usually $\NN$, is called a Bessel sequence for the Hilbert space $(\Hil, \<\cdot,\cdot\>)$
 if there exists a  constant ${ \sf M}\in(0,\infty)$ such that
\be
  \sum _{k\in \Gamma} \left| \ip {\psi_k}{f} \right|^2  \le { \sf M}  \norm{}{f}^2 ,  \forall \, f \in \Hil.
\label{eq:discr-bessel}
\end{equation}
To a given Bessel sequence $\Psi$ for $\Hil$, the following three operators are associated:
\bei
\item The analysis operator $C: \Hil \rightarrow  \ell  ^2 $ given by  $Cf = \{\ip {\psi_k }{f }, k \in \Gamma\}$, which is the analogue
 of $C_{\Psi}: \Hil \to L^{2}(X, \ud\nu )$ ;  

\item The synthesis operator  $D : \ell  ^2   \rightarrow \Hil$  given by
 $Dc = \sum_k c_k \psi_k,$ where $c= (c_k)$;

\item  The frame operator $S:\Hil\to\Hil$ given by   
 $Sf = \sum_k \ip{\psi_k}{f}\, \psi_k$,  so that
$$
 \ip{f}{Sf}\ = \sum _{k\in \Gamma} \left| \ip {\psi_k}{f} \right|^2.
$$
\eni
Moreover,   we have $D=C^\ast, \; C=D^ \ast $, and  $S= C^ \ast C$.

 \subsection{Discrete frames}
\label{subsec:bdd-discrframe}

 For the sake of completeness and comparison with the unbounded case, it is worthwhile to quickly summarize the salient features of 
 frames, following closely Section \ref{subsec:bdd-frames}.  We do it in the form of a theorem. Of course, all the statements below are
 well-known \cite{Casaz1,ole1,Groech1}, but the approach is non-standard following the continuous approach \cite{jpa-sqintegI,jpa-contframes,jpa-CSbook}.
\begin{theorem} 
\label{theor-41}
Let $\Psi=(\psi_k)$ be a frame in $\Hil$,
with analysis operator $C : \Hil \rightarrow \ell^2 $, synthesis operator $D: \ell  ^2   \rightarrow \Hil$ and frame operator is $S: \Hil \rightarrow \Hil $.
Then
\bei
\item[(1)] $\Psi  $ is total in $\Hil$. The operator $S$ has a bounded inverse $S^{-1}:\Hil\to\Hil$ and one has the reconstruction formulas 
\begin{align} \label{srepr}
f &=  S^{-1} S f = \sum _k \ip{\psi_k}{ f}  S^{-1} \psi_k , \; \mbox{ for every }\; f\in \Hil, 
\\ \label{srepr2}
f &=   S S^{-1} f = \sum _k \ip{S^{-1} \psi_k}{ f}   \psi_k , \; \mbox{ for every }\; f\in \Hil. 
\end{align}

\vspace*{-2mm} \item [(2)] $R_{C}$ is a closed subspace of $\ell^2$. The projection $P_\Psi$ from $\ell^2$ onto  $R_{C}$ is given by $P_\Psi = C S^{-1} D$
 $=C C^+$, where, as before, $C^+$ is the pseudo-inverse of $C$.   

\item [(3)] Define
\be
\ip {c}{d }_{\Psi} = \ip {c}{C S^{-1} C^{-1} d} _{\ell  ^2 }, \,  \, c, d\in R_C. 
\label{eq:discr-scalarprod}
\end{equation}
The relation (\ref{eq:discr-scalarprod}) defines an inner product on $R_C$ and $R_C$ is complete in this inner product. 
Thus, $(R_{C}, \<\cdot,\cdot\>_\Psi)$ is a Hilbert space, which will be denoted by $\h_\Psi $.

\item [(4)] $\h_\Psi $ is a reproducing kernel Hilbert space with kernel given by the matrix  ${\mathcal P}_{k,l} = \ip  { \psi_k}   {S^{-1}\psi_l }$. 

\item [(5)]  The analysis operator $C$ is a unitary operator from $\Hil$ onto  $\h_\Psi $.
Thus, it can be inverted on its range by its adjoint, which leads to the reconstruction formula (\ref{srepr}).  
\eni
\end{theorem} 
\begin{proof} 
(1) follows from the lower frame condition.   

 (2) \& (3)  $C$ is defined on all of $\Hil $, $S = D C$, and so $C$ is injective and  $D$ is surjective. Therefore, for all  $f,f'\in \Hil $,
$$
 \ip {C f}{Cf'}_{\Psi}=  \ip {C f} {C S^{-1} C^{-1} C f}_{\ell^2 } =\ip {C f} {C S^{-1} f}_{\ell^2 }
= \ip {  f} {DC S^{-1} f}_{\Hil} = \ip {f}{f'}_{\Hil}.
$$
Thus $C$  is isometric from $\Hil$ onto  $R_C$, equipped with the new inner product \eqref{eq:discr-scalarprod}. Thus the latter is complete, 
hence  a Hilbert space. 

  Since $S^{-1}$ is bounded, the norms $\norm{\ell^2 }{\cdot}$ and $\norm{\Psi}{\cdot}$ are equivalent. Therefore, $R_{C}$ is a closed subspace of $\ell^2$.
Since $C$  is an isometry onto $\h_\Psi $, the corresponding projection onto $R_C = \h_\Psi $ is $C{C^*}^{(\Psi)}$, where ${C^*}^{(\Psi)}:\h_\Psi\to \Hil$ 
denotes the adjoint  of $C: \Hil \to \h_\Psi$, and in fact coincides with the pseudo-inverse $C^+$ of $C$. Then we have,
 for every $c\in\h_\Psi $ and every $ f\in \Hil$,
$$
 \ip{ {C^*}^{(\Psi)} c }  {f}_{\Hil} = \ip {c} { C f } _ \Psi 
=  \ip {c}  { C S^{-1} f} _{\ell^2 } = \ip {\left( C S^{-1} \right)^{*} c}  {f} _{\Hil} = \ip { S^{-1} C^* c}  { f} _{\Hil} = \ip { S^{-1} D c}  { f} _{\Hil} .
$$
So $C^{*(\Psi)} = S^{-1} D_{\up_{\h_\Psi}} =C^+$ and $P_\Psi = C S^{-1} D : \ell^2 \to R_C$. One verifies immediately that $P_\Psi^2 =P_\Psi$ and $P_\Psi^\ast = P_\Psi$.

(4) Given $c\in \h_\Psi$, we have $ S^{-1}Dc = \sum_k (S^{-1}\psi_k) c_k$ and thus 
$$
(P_\Psi c)_l = (CS^{-1}Dc)_l  =  \sum_k (CS^{-1}\psi_k)_l c_k = \sum_k \ip{\psi_l }{S^{-1}\psi_k}c_k = \sum_k{\mathcal P}_{l,k}c_k.
$$

(5) Obvious. 
\end{proof}
\medskip

Since   the norms $\norm{\ell^2 }{\cdot}$ and $\norm{\Psi}{\cdot}$ are equivalent, 
 here  the three Hilbert spaces of    the Gel'fand triplet \eqref{eq:triplet} coincide as sets, with equivalent norms.

\subsection{Discrete upper semi-frames}
\label{subsec:unbdd-discrframe}

Let now $\Psi$ be an upper semi-frame, that is, a sequence $(\psi_k)$ satisfying the relation
\be
0 < \sum_{k\in \Gamma} \left| \ip {\psi_k}{f} \right|^2  \le {\rm\sf M}  \norm{}{f}^2 ,  \forall \, f \in \Hil, \, f\neq 0.
\label{eq:discr-unbddframe}
\end{equation}
First we have the following easy result, that follows immediately from \eqref{eq:discr-unbddframe}.
\begin{lem.}
$\Psi = (\psi_k)$ is an upper semi-frame if and only if it is a  total Bessel sequence.
\end{lem.}
\medskip

As we recalled in Section \ref{subsec:bdd-discrframe}, a useful property of a frame $\Phi$ for $\Hil$ is that every element in $\Hil$
 can be represented as series expansions in the form
\begin{equation} \label{ff}
f=\sum \ip {\phi_k}{f}\, \psi_k = \sum \ip {\psi_k }{f}\,  \phi_k
\end{equation}
via some sequence $\Psi$. However, there exist Bessel sequences $\Phi$ for $\Hil$, which are not frames and for which \eqref{ff} holds 
via a sequence $\Psi$ (which cannot be Bessel for $\Hil$); for example, consider $\Phi=(\frac{1}{k}e_k)$ and $\Psi=(ke_k)$ (see Section \ref{subsec:simplexamples}). 
Thus, the frame property is sufficient, but not necessary for series expansions in the form \eqref{ff}.
As a matter of fact, if one requires a series expansion via a Bessel sequence  which is not a frame, then the dual sequence cannot be Bessel \cite{jpa-Besseq}.

Before going into these duality considerations, we analyze the various operators, as in the general case.

\begin{lem.}
 Let   $\Psi$ be an upper semi-frame. Then one has:
\bei
\item[(1)] The analysis operator $C$ is an injective bounded operator and the synthesis operator $D$ is a bounded operator with dense range.
The frame operator $S$ is a bounded, self-adjoint, positive operator with dense range. 
Its inverse $S^{-1}$ is  densely defined and self-adjoint.

\item[(2)] $R_C^{\Psi} \subseteq R_C \subseteq \ov{R_C}$, with dense inclusions, where $R_C^{\Psi}:= C(R_{S})$ and
$\ov{R_C}$ denotes the closure of $R_C$ in $ \ell  ^2  $.
\eni
\end{lem.}
\begin{proof}
(1)  Since $(\psi_k)$ is   a total    Bessel sequence,   the operators $C,D,S$ are bounded. $C$ is clearly injective, so $D$ has dense range
and $S$ reads, with unconditional convergence,
$$
Sf = \sum_k \ip{\psi_k}{f}\, \psi_k, \; \mbox{ for all } \; f\in \Hil.
$$
{As $\ov{R_C} = {\sf Ker}(D)^\bot$, $S$ in injective.}
Since $S = D C = C^* C = D D^* ,  \,S$ is self-adjoint and positive. $R_S$ is dense and $S^{-1}$ is self-adjoint and
positive, with dense domain $\dom(S^{-1})= R_S \subset \Hil$. 

(2) This is immediate.
\end{proof}
\medskip

In accordance with the continuous case, we say that the upper semi-frame $\Psi = (\psi_k)$ is \emph{regular} if
 every $\psi_{k}$ belongs to $\dom(S^{-1})=R_S$. First note that, if $\Psi$ is an upper semi-frame for $\Hil$, then 
 \begin{equation} \label{ssminus1}
 f=SS^{-1}f=\sum_{k} \ip{\psi_k}{S^{-1}f}\psi_k, \ \forall f\in R_S.
 \end{equation}
If we want to write the expansion above using a dual sequence (similar to the frame expansion \ref{srepr2}), then the upper semi-frame should be regular.

\begin{prop.} \label{ssminus} 
Let $\Psi$ be a regular upper semi-frame for $\Hil$. Then
\begin{equation}\label{eq:trivreconstunb1}
f=SS^{-1}f=\sum_{k} \ip{S^{-1}\psi_k}{f}\psi_k, \ \forall f\in R_S.
\end{equation}
\end{prop.} 
\begin{proof} As in the continuous case, this follows from the facts that $S$ is bounded and $S^{-1}$ is self-adjoint.
\end{proof}
\medskip

\berem
It does not seem possible to extend this reconstruction formula to all $f\in\Hil$.
Indeed, let $ R_S \ni f_{n}\to f \in \Hil $. Then we have d 
 $\ip{S^{-1}\psi_k}{f_{n}}_{\Hil} \to \ip{S^{-1}\psi_k}{f}_{\Hil}$ pointwise. 
If we knew in addition that  $\ip{S^{-1}\psi_k}{f}_{\Hil}$ does belong to
$\ell^2$, then  the r.h.s. of \eqref{eq:trivreconstunb1} would be  an inner product in $\ell ^2$,
 hence continuous in both terms, so that we could conclude that \eqref{eq:trivreconstunb1} is valid for every $f \in \Hil$.  

However, we can show that
$\ip{S^{-1}\psi_k}{f}_{\Hil} \in \ell^2$ for all $f$ if and only if $S^{-1}$ is bounded.
Indeed, if $S^{-1}$ is bounded, then $(S^{-1}\psi_k)$ is Bessel, so this direction is clear.
On the other hand, let $\psi_k$ be a total Bessel sequence. 
As $\ip{S^{-1}\psi_k}{f}_{\Hil} \in \ell^2$ we have that $\phi_k = S^{-1}\psi_k$ is a Bessel sequence.
 Let indeed $Cf = (\ip{\phi_k}{f})\in \ell^2, \forall f $,  i.e. $\dom(C)$ is the whole Hilbert space. By \cite[Prop. 4.1 (a1)]{BSA}  $ (\phi_k) $ is Bessel.  
Then, by \cite[Prop. 4.6, (a) and (g)]{BSA}, $\psi_k = V e_k$ with $V$ bounded and having dense range. 
As $\phi_k = S^{-1} Ve_k$ is a Bessel sequence by the same result, $S^{-1} V$ must be bounded. Therefore $S^{-1}$ is bounded on $R_V$. But as $R_V$ is dense, $S^{-1}$ 
can be extended to a bounded operator everywhere.

Note this implies that, if the reconstruction formula can be extended to the whole Hilbert space by the strategy just described, then the original sequence was already a frame. 
\enrem
\medskip

The whole motivation of the present construction is to translate abstract statements in $\Hil$ into concrete ones about sequences,
 taking place in $\ell^2$. The correspondence is implemented by the operators $C$ and $C^{-1}$. Hence we first transport
the operators $S$ and $S^{-1}$, according to the following proposition.

\begin{prop.} \label{prop44}
\bei\item[(1)] Define the operator $G_{S} : R_C \to R_C^ \Psi $   by  $G_{S} = C   S Ê  C^{-1}$  
and the operator $G_{S} ^{-1} : R_C^ \Psi \to R_C$   by  $G_{S} ^{-1} = C   S^{-1} Ê  C^{-1}{|_{R_C^\Psi}}.$ Then, in the Hilbert space $ \ov{R_C}$,
$G_{S} $ is a bounded, positive and symmetric operator, while $G_{S} ^{-1}$ is positive and essentially self-adjoint.
These two operators are bijective and inverse to each other.

\item[(2)]  Let $G = \ov{G_{S}}$ and let $G^{-1}$ be the self-adjoint extension of  $G_{S} ^{-1}$.
Then $G: \ov{R_C} \to R_{G} \subseteq \ov{R_C}$ is 
bounded, self-adjoint and positive  with  $\dom(G) = \ov{R_C}$, thus $G = C D{|_{ \ov{R_C}}}$.
Furthermore $G^{-1}: \dom(G^{-1}) \subset \ov{R_C} \to \ov{R_C}$ is self-adjoint and positive, with domain 
$ \dom(G^{-1}) = R_{G}=   C (R_{D})  $, a dense subspace of $ \ov{R_C}$. The two operators are inverse of each other, in the sense that
$$
 G^{-1}G = I{|_{\ov{R_C}}}, \quad G G^{-1}=  I{|_{C(R_D)}}.
$$
\eni
\end{prop.}
\begin{proof}
(1) $G_{S} $ and $G_{S} ^{-1}$ are well-defined on their respective domains and obviously    inverse to each other.
Let $c \in R_C$ and $d \in R_C^ \Psi $ with $c=C f$ and $d = C g$ for $f \in \Hil$ and $g \in R_S$.
Then we have
$$
\ip {c}   {G_{S} ^{-1} d} _{\ell^2 } =
 \ip { c} {C S^{-1} C^{-1} d}_{\ell^2 } = \ip {C f}{C S^{-1} g}_{\ell^2 } =\ip { f}{D C S^{-1} g}_{\Hil} 
 = \ip { f }{g }_{\Hil} = \ip { C^{-1} c}{C^{-1} d}_{\Hil}.
$$
Therefore $G_S$ is positive.
Clearly   $S C^{-1}|_{R_C} = D|_{R_C}$ and so $G_S$ is bounded.
$G_{S} ^{-1} = {C^{-1}}^* C^{-1}|_{R_C^ \Psi}$, so  it is symmetric, and therefore closable.
  Furthermore ${G_{S} ^{-1}}^* = {C^{-1}}^* C^{-1}|_{R_C}$.
It remains to show that the defect indices of $G_{S} ^{-1}$ are (0,0), which results from a direct calculation. Indeed, let $z\in\CN$ with Im\,$z \neq 0$ and
suppose there  exists a element $c\in R_{C}$ such that 
\be\label{eq:defect}
\ip{c}{(G_{S} ^{-1} - z)d}_{\ell^2}=0, \mbox{ for all }  d\in R_C^ \Psi.
\end{equation}
Then, if $c= C f, \, d= Cg$, with $f,g\in\Hil$, it follows that \eqref{eq:defect} implies
$$
\ip{f}{g} = z \ip{f}{Sg}, \; \forall \, g\in\Hil.
$$
By the positivity of $S$, we must have $f=0$ and therefore $c=0$. Since $R_{C}$ is dense in $\ov{R_{C}}$, this implies that $G_{S} ^{-1}$  has  defect indices (0,0) 
 and thus is 
essentially self-adjoint.

(2) follows immediately from (1). The fact that $G$ and $G^{-1}$ are inverse of each other follows from the corresponding relation of 
 $G_{S} $ and $G_{S} ^{-1}$ and the definition of the domain of the closure of an operator, namely
 $ \dom(G^{-1}) = \{c \in \ov{R_{C}}: c=\lim_{i} c_{i}, \,c_{i}\in R_{C},
\,\mathrm{and}\, (G_{S}^{-1}c_{i}) \,\mathrm{converges}\}$. 

First, for every $c\in \ov{R_{C}}, \, Gc\in \dom(G^{-1})$. Take indeed
$\ov{R_{C}}\ni c = \lim_{i}c_{i}, \,c_{i}\in R_{C}$. Then $Gc = \lim_{i}G_{S} c_{i}$ and $G^{-1}Gc_{i} = G^{-1}G_{S} c_{i}  
= G^{-1}_{S}G_{S} c_{i} = c_{i} \to c$.
Next $G^{-1}Gc =G^{-1}G \lim_{i}c_{i}= G^{-1}\lim_{i}G_{S}c_{i}=  \lim_{i}G_{S}^{-1}G_{S}c_{i} = c$.

On the other hand, for every $c\in \dom(G^{-1})$, one has $GG^{-1}c = GG^{-1}\lim_{i}c_{i} = G\lim_{i}G_{S}^{-1} c_{i} 
= \lim_{i}G_{S} G_{S}^{-1} c_{i} = c$.
This proves that indeed $G$ and $G^{-1}$ are inverse to each other on the appropriate domains.

Finally,  $G^{-1}$ is  positive, since its inverse $G$ is positive and bounded, and thus the spectrum of $G^{-1}$ is bounded away from 0.
 \end{proof}
\bigskip

 {\bf Remark}:  the proof of Proposition \ref{prop44} is  partly  modelled on the similar one in the original paper \cite{jpa-sqintegI}.
\medskip

Putting everything together, we have  the following commutative diagram: 
\be\label{diagram3}
\hspace*{-1cm}\begin{minipage}{15cm}
\begin{center}
\setlength{\unitlength}{2pt}
\begin{picture}(100,60)
\put(10,50){$\Hil$}
\put(89,50){$R_C \subseteq  \ov{R_C}\subseteq \ell^2 $}
\put(20,52){\vector(4,0){65}}
\put(54,44){$C^{-1}$}
\put(85,50){\vector(-4,0){65}}
\put(55,54){$C$}
\put(10,47){\vector(0,-4){30}}
Ê\put(12,17){\vector(0,4){30}}
\put(91,47){\vector(0,-4){30}}
Ê\put(93,17){\vector(0,4){30}}
\put(-35,10){$ \Hil \supseteq \dom(S^{-1})=R_S$}
\put(85,47){\vector(-2,-1){65}}
\put(14,30){$S^{-1}$}
\put(5,30){$S$}
\put(82,30){$G_{S}$}
\put(95,30){$G_{S} ^{-1}$}
\put(89,10){$R_C^ \Psi \subseteq \ell^2$}
\put(20,12){\vector(4,0){65}}
\put(50,4){$C^{-1}$}
\put(48,33){$D$}
Ê\put(85,10){\vector(-4,0){65}}
\put(50,15){$C$}
\end{picture}
\end{center}
\end{minipage}
\end{equation}

As before, define on $R_{C}^{\Psi}$ the new inner product $\ip {c}{d}_ \Psi = \ip{c}{G^{-1} d}_{\ell^2 }$, which makes sense
since $G^{-1}$ is self-adjoint and positive. Therefore $\ip {c}{d}_ \Psi = \ip{G^{-1/2}c}{G^{-1/2} d}_{\ell^2 }$.
Denote by
${\h}_{\Psi}:={\ov{R_{C}^{\Psi}}}^{\Psi} $ the closure of $R_{C}^{\Psi}$ in the corresponding new norm, which is a Hilbert space.

Then the fundamental result reads as follows.
\begin{theorem}
\label{theor-43} Let ${\h}_{\Psi} := {\ov{R_{C}^{\Psi}}}^{\Psi}$ be defined as above. 
Then: 
\bei
\item[(1)] ${\h}_{\Psi}$ coincides with $R_C$ (as sets) and $C$ is a unitary map  (isomorphism) between $\Hil$ and ${\h}_{\Psi}$.

\item[(2)] 
The norm $\norm{\Psi}{.}$ is equivalent to the graph norm of $G^{-1/2}$ and, therefore, $\dom( G^{-1/2} ) = {\h}_{\Psi}$. 

\item[(3)]
  $C^{*(\Psi)} = \left( S^{-1} D \right)\!{|_{{\h}_{\Psi}}}$, where $C^{*(\Psi)}:{\h}_{\Psi} \to \Hil$ is the adjoint of
$C: \Hil  \to{\h}_{\Psi} $. Moreover,  for every $f\in\Hil$, one has
$$
 f = C^{*(\Psi)} C f = \left( S^{-1} D \right) C f. 
$$
\item[(4)] 
For all $f \in R_D$, we have the reconstruction formula
 \begin{equation}\label{sec:reconssynth2}
  f = \sum \limits_k \left[ G^{-1}  \left( \left< f , \psi_k \right>_\Hil \right) \right] \psi_k 
  \vspace*{-5mm}\end{equation}
with unconditional convergence.
\eni
\end{theorem} 
\begin{proof} (1) Performing the same calculation as before, we see that  the map $C$, restricted to the dense domain  $\dom(S^{-1}) =   R_{S}$,  
is an  {isometry} into  ${\h}_{\Psi}$ :
$$
 \ip {C f}{C g}_{\Psi}=\ip {f}{g}_{\Hil}, \; \forall\, f,g\in  R_{S}.
 $$
Thus $C$ extends by continuity to a  unitary map from $\Hil= \ov{R_{S}}$ onto  ${\h}_{\Psi}={\ov{R_{C}^{\Psi}}}^{\Psi} $.
As $C$ is an isometric isomorphism from $\Hil$ onto $\h_{\Psi}$ and is bounded from $\Hil$ onto $R_C$, we have ${\h}_{\Psi} =   R_{C}$.

(2) By definition,   $\norm{\Psi}{c}^2 = \ip{c}{G^{-1} c}_{\ell^2 }= \ip{G^{-1/2} c}{G^{-1/2} d}_{\ell^2 }$ for every $c,d\in  R_C^ \Psi$. 
Since $G^{-1}$ 
is  self-adjoint and positive, and has a bounded inverse, its spectrum is bounded away from 0, hence  $\norm{\Psi}{.}$ is equivalent to the graph norm 
of $G^{-1/2} $, which implies that  $\dom( G^{-1/2} ) = {\h}_{\Psi}$

(3) Since the operator $C: \Hil \to  {\h}_{\Psi}$ is unitary, it can be inverted on ${\h}_{\Psi}$ by its adjoint $C^{*(\Psi)}$, 
which yields the formula  $f = C^{*(\Psi)} C f$, $\forall f\in\Hil$. 
 Furthermore, for every $f\in\Hil$ one has $f=S^{-1}Sf=S^{-1} D  C f$.   

(4)  As $f \in R_D$, $C f \in C(R_D) = R_G = \dom (G^{-1})$ and thus the composition $G^{-1} C$ is well defined. 
Furthermore, 
there exists a $c \in \ov{R_C} = {\sf Ker}(D)^\bot$ with $f = D  c$. So 
\begin{align*}
 \sum_k \left[ G^{-1}  \left( \ip {\psi_k }{f} _\Hil \right) \right] \psi_k = D G^{-1} C f = D G^{-1} C D c = D G^{-1} G c = D c = f.
\\[-18mm]
\makebox[3cm]{}\end{align*}
\end{proof}

\vspace*{5mm}
Exactly as in the continuous case of Section \ref{subsec:unbdd-frames}, we have the following diagram that particularizes \eqref{diagram2}:
\be\label{diagram4}
\begin{array}{cccc}
\Hil       &\stackrel{C}{\longrightarrow }&{\h}_{\Psi}= R_{C} \subset & \!\!\!\!\ov{R_{C}} \subset \ell^{2} \\[1mm]
\cup&                            &\cup& \\[1mm]
\dom(S^{-1}) = R_{S}   &\stackrel{C}{\longrightarrow }  &\qquad R_{C}^{\Psi}\quad \subset & \hspace{-3mm}  \ell^{2}
\end{array}
\end{equation}

We know $G: \ov{R_C} \rightarrow \ov{R_C}$ is bounded and non-negative with $R_G = C(R_D)$, so $G^{1/2}: \ov{R_C} \rightarrow \ov{R_C}$ 
is bounded and non-negative,  with the same kernel.

\begin{cor.} \label{cor:G±1/2}
${G^{1/2}}: \ov{R_C} \rightarrow {\h}_{\Psi}$ is a unitary operator (isomorphism) and so is its inverse 
${G^{-1/2}} : Ê{\h}_{\Psi} \rightarrow \ov{R_C}$. Both operators are positive.

\end{cor.}
\begin{proof} 
As $G$ is positive (by Prop. \ref{prop44} (2)), $G^{1/2}$ is positive, too.
Since $G$ is bounded, the domain of its square root is also $\ov{R_C}$.  Since $\ran({G}^{1/2})$ is closed,  
$ {G}^{1/2} : \ov{R_C} \rightarrow {\h}_{\Psi}$ is unitary. 
Its inverse $ {G}^{-1/2} :{\h}_{\Psi} \rightarrow \ov{R_C}$ is therefore also unitary. 
\end{proof}

\begin{prop.}\label{sec:rkhsunbound1} Let $(\psi_k)$ be a regular upper semi-frame.
Then ${\h}_{\Psi}$ is a reproducing kernel Hilbert space, with kernel given by  the operator $S^{-1}D $, which is a matrix operator,
 namely,  the matrix $\mathcal G$, where  $\mathcal G_{k,l} = \ip{Ê\psi_k}   {S^{-1} \psi_l}   = \ip {\psi_k}  { C^{-1} G^{-1} C Ê\psi_l}  $. 

\end{prop.}
\begin{proof} 
Let $d = Cf\in {\h}_{\Psi}$. Then, as   $Dd\in R(D|_{R_C}) \subset R(S) = \dom(S^{-1})$, we have  

\begin{align*}
\sum_l {\mathcal G}_{k,l} d_l &=\sum_l  \ip{\psi_k}   {S^{-1} \psi_l } d_l = \ip{ÊS^{-1}\psi_k}{\sum_l \psi_l  d_l }= 
\ip{ÊS^{-1}\psi_k}{Dd} =  \ip{Ê\psi_k}{S^{-1}Dd}\\
&=(CS^{-1}D d)_k = (CS^{-1}D Cf)_k =  (Cf)_k = d_k.
\end{align*}
\end{proof}

For $f \in R_S$ we have $f = S S^{-1} f$. So, for a regular upper semi-frame, we can give the reconstruction formula for all $f \in R_S$
$$ f = \sum \left< f , \widetilde \psi_k\right> \psi_k,
 $$ 
where, as usual, $\widetilde \psi_k:= S^{-1} \psi_k$ denotes the canonical dual.

From the results above, we know that 
$G$ is a bounded, positive and bijective operator from $\ov{R_C}$ onto $C(R_D)$. Furthermore $G^{1/2}$ maps $\ov{R_C}$ bijectively onto $R_C$. 
This means that $G^{1/2}$ also maps $R_C$ bijectively on $C(R_D)$.
As  $G C = C D C = C S$, we now know that $G$ maps $R_C$ bijectively onto $C(R_S)$ and so $G^{1/2}$ maps $C(R_D)$ bijectively onto $C(R_S)$.
In summary, we have
\begin{equation}
\label{sec:GGGdia1}
\begin{array}{ccccccccc}
\ov{R_C}       & \stackrel{G^{1/2}}{\longrightarrow } & R_{C}        & \stackrel{G^{1/2}}{\longrightarrow } & C(R_D) & \stackrel{G^{1/2}}{\longrightarrow } & C(R_S), 
\end{array}
\end{equation}
where each operator $G^{1/2}$ is a bijection.

Clearly $\left(D G^{-1} C \right){|_{R_S}} = I{|_{R_S}}$. 
Furthermore as  $D G^{-1/2} C D G^{-1/2} C = D G^{-1/2} G G^{-1/2} C = D C = S$ and $D G^{-1/2} C$ is clearly a positive operator, we have
\begin{equation} \label{sec:rootSG1}
 S^{1/2} = D G^{-1/2} C. 
 \end{equation}
Therefore $S^{1/2}$ maps $\Hil$ bijectively onto $D(\ov{R_C}) = R_D$. And so $S^{1/2}$ maps $R_D$ bijectively onto $R_S$.

With similar arguments, we can show that the following diagram is commutative: 
\be\label{gramframconn1}
\hspace*{-1cm}\begin{minipage}{15cm}
\begin{center}
\setlength{\unitlength}{2pt}
\begin{picture}(100,60)
\put(10,15){$\Hil$}
\put(25,15){$\stackrel{S^{1/2}}{\longrightarrow }$}
\put(40,15){$ R_D$}
\put(55,15){$\stackrel{S^{1/2}}{\longrightarrow }$}
\put(72,15){$R_S$}
\put(90,15){$\stackrel{S^{1/2}}{\longrightarrow }$}
\put(105,15){$S(R_D)$}
\put(10,50){$\ov{R_C}$}  
\put(25,50){$\stackrel{G^{1/2}}{\longrightarrow }$}
\put(40,50){$ R_C$}
\put(55,50){$\stackrel{G^{1/2}}{\longrightarrow }$}
\put(70,50){$ C(R_D)$}
\put(90,50){$\stackrel{G^{1/2}}{\longrightarrow }$}
\put(105,50){$C(R_S)$}
\put(16,45){\vector(1,-1){23}}
\put(46,45){\vector(1,-1){23}}
\put(81,45){\vector(1,-1){23}}
\put(16,22){\vector(1,1){23}}
\put(46,22){\vector(1,1){23}}
\put(81,22){\vector(1,1){23}}
\put(17,35){$\scriptstyle D$}
\put(47,35){$\scriptstyle D$}
\put(82,35){$\scriptstyle D$}
\put(34,35){$\scriptstyle C$}
\put(64,35){$\scriptstyle C$}
\put(99,35){$\scriptstyle C$}
\end{picture}
\end{center}
\end{minipage}
\vspace*{-5mm}
\end{equation}

\noi The reconstruction formula   given in Theorem \ref{theor-43}(4) is valid for every $f\in R_D$. The one in \eqref{ssminus1} is valid 
for every $f\in R_S$.  
With the results given above,  we can give a reconstruction formula valid for all $f \in \Hil$, even in the case  when $\Psi \not\subseteq \dom(S^{-1})$, 
if we allow the analysis coefficents to be altered. 
 
\begin{theorem}\label{sec:reconssynthfull1} 
Let $(\psi_k)$ be an upper semi-frame.
Then, for all $f \in \Hil$, we have the reconstruction formula
$$
 f = S^{-1/2} \sum _k \left[ G^{-1/2} \ip{ \psi_k} {f} \right]\psi_k .$$
\end{theorem}
\begin{proof} By eq. \eqref{sec:rootSG1},
$$ S^{-1/2} \sum _k \left[ G^{-1/2} \ip{ \psi_k} {f} \right]\psi_k = S^{-1/2} D G^{-1/2} C f = S^{-1/2} S^{1/2} f = f. $$
\end{proof}

For applications and implementations, this is not a very `useful' approach, since it uses an operator-based approach and does not use sequences for the inversion.
For a treatment of the existence of dual sequences  and related questions, we refer to \cite{jpa-Besseq}.

\subsection{Formulation in terms of a Gel'fand triplet}
\label{subsec:discG-triplet}
 
As in the continuous case, the discrete setup may also advantageously be formulated with a  triplet of Hilbert spaces, namely,
\be
 {\h}_{\Psi} \;\subset\; {\h}_{0}= \ov{{\h}_{\Psi}}\;\subset \; {\h}_{\Psi}^{\times}
\label{eq:disctriplet}
\end{equation}
where  
\begin{itemize}
\vspace{-1mm}\item[{\bf .}]  $ {\h}_{\Psi} = R_{C}$, which is a Hilbert space for the norm  $\norm{\Psi}{\cdot}=\ip {\cdot}{C  S^{-1} C^{-1} \cdot }^{1/2} = 
\ip {\cdot}{G^{-1} \cdot }^{1/2}$;
\vspace{-1mm}\item[{\bf .}]   $\ov{{\h}_{\Psi}}$ is the closure of ${\h}_{\Psi}$ in $\ell^{2}$, that is, $\ov{R_C} $;
\vspace{-1mm}\item[{\bf .}]  ${\h}_{\Psi}^{\times}$ is the conjugate dual of ${\h}_{\Psi}$ and the completion of  ${\h}_{0}$ in the norm
$\norm{\Psi}{\cdot}^{\times}:=\ip {\cdot}{CSC^{-1} \cdot }^{1/2} = \ip {\cdot}{G \cdot }^{1/2} $.  
\end{itemize}
 Then, if $\psi_{k}\in \dom(S^{-1}), \, \forall\, k $, all three spaces $ {\h}_{\Psi},{\h}_{0}, {\h}_{\Psi}^{\times}$ are  {reproducing kernel Hilbert spaces}, with the same 
(matrix) kernel  
${\mathcal G}(k,l) = \ip{\psi_{k} }{S^{-1}\psi_{l}}.$
Here too,  ${\h}_{\Psi}^{\times}$ carries the unbounded version of the dual frame.
\medskip

In the triplet \eqref{eq:disctriplet}, ${\h}_{0}$ is a sequence space contained in $\ell^2$ (possibly $\ell^2$ itself), ${\h}_{\Psi}$ is a smaller sequence space, for instance a space
of \emph{decreasing} sequences. But  ${\h}_{\Psi}^{\times}$ is the K\"{o}the dual of ${\h}_{\Psi}$, normally not contained in $\ell^2$. In the example below, 
${\h}_{\Psi}^{\times}$ consists of slowly \emph{increasing} sequences.

 \subsection{Lower semi-frames and duality}
\label{subsec:simplexamples} 

{As announced in Section \ref{subsec:upperlower}, the notion of upper semi-frame has a dual, that of lower semi-frame.
Thus, particularizing \eqref{eq:lowersf}, we say that  a    sequence $\Phi=\{\phi_{k}\}$
 is a \emph{lower semi-frame} if  it satisfies the lower frame condition, that is, there exists   a constant ${\sf m}>0$ such that
\be
{\sf m}  \norm{}{f}^2 \le  \sum _k |\ip{\phi_{k}}{f}| ^2 \, , \;\; \forall \, f \in \Hil.
\label{eq:disclowersf} 
\end{equation}
Clearly, \eqref{eq:disclowersf} implies that the family $\Phi $ is total in $\Hil$. Notice there is a slight dissimilarity between the two definitions of semi-frames.
In the upper case \eqref{eq:discr-unbddframe}, the positivity requirement on the left-hand side ensures that the sequence $\Psi$ is total, whereas here, it follows automatically from the lower frame bound. 
Before exploring further the duality between the two notions, let us give some simple examples.}

 Let $(e_k)$ be an orthonormal basis in $\Hil$   with index set $\mathbb{N}$ 
 (we have to stick to infinite-dimensional spaces, since every sequence in $\CC^N$  is a frame sequence, so there are no upper semi-frames which are not frames). 
 Let $\psi_k = \frac{1}{k} e_k$. Then $(\psi_k)$ is an upper semi-frame :
$$
0 <  \sum_k | \ip{\psi_k }{f}|^2 \leq \sum _k | \ip{e_k }{f}|^2 = \|f\|^2.
$$
Indeed, there is no lower frame bound, because for $f=e_{p}$, one has $\sum \limits_k | \ip{\psi_k }{f}|^2=\frac{1}{p^2}$.

  Let  $\phi_{k}= k\, e_{k}$. The sequence $(\phi_{k})$ is dual to $(\psi_k)$, since one has
$$
\sum_k \ip{\psi_k }{f} \phi_{k} = f.
$$
In addition, we have
$$
\sum _k | \ip{e_k }{f}|^2 = \|f\|^2 \leq \sum_k | \ip{\phi_k }{f}|^2,
$$
and this is unbounded since $\sum_k | \ip{\phi_k }{f}|^2 = p^2$ for $f=e_{p}$. Hence,  $(\phi_{k})$ is a lower semi-frame, dual to  $(\psi_k)$, and it lives also 
in ${\h}_{\Psi}$. 

  In this case, the frame operator associated to $(\psi_k)$  is $S= \mathrm{diag}(\frac{1}{n^2})$. Thus $S^{-1}= \mathrm{diag}({n^2})$, which is clearly unbounded,
and the inner products are, respectively :
\begin{itemize}
\vspace{-1mm}\item[{\bf .}] For ${\h}_{\Psi}: \quad\ip{c}{d}_{\Psi} = \sum_{n}{n^2}\, \ov{c_{n}}\,d_{n}$; 
\vspace{-1mm}\item[{\bf .}]   For ${\h}_0: \quad\;\ip{c}{d}_{0} =\sum_{n}  \ov{c_{n}}\,d_{n}$; 
\vspace{-1mm}\item[{\bf .}]   For ${\h}_{\Psi}^{\times}: \quad\ip{c}{d}_{\Psi}^{\times} =\sum_{n} \frac{1}{n^2} \,\ov{c_{n}}\,d_{n}$. 
 \end{itemize}
It follows that $(\phi_{k})$ is the canonical dual of $(\psi_k)$, since $\phi_{k} = S^{-1} \psi_k $.
 The sequence used by Gabor  in his original IEE-paper \cite{gabor}, a Gabor system with a Gaussian window, $a=1$ and $b=1$, is exactly 
such an upper semi-frame.  This example has been analyzed in great detail by Lyubarskii and Seip \cite{lyu-seip}. Interestingly enough,
the technical tool used there is a scale of Hilbert spaces interpolating between the Schwartz spaces $\mathcal S$ and  $\mathcal S'$, one of the simplest
examples of partial inner product spaces \cite{jpa-pipspaces}.

The example $(\frac{1}{k}e_k), (ke_k)$ can be generalized to weighted  sequences   $(m_k e_k)$, for adequate weights $m_{k}$. We refer to
\cite{jpa-Besseq,stobal09,bstable09} for additional information.
\medskip

In the case of discrete semi-frames, the duality between upper and lower ones has been studied in several papers, e.g. \cite{casoleli1}.
Here we simply note two results.
First, we have the equivalent of Lemma \ref{lem:analop}:

\belem   {\rm \cite[Lemma 3.1]{casoleli1}}
Given any total family $\Phi =\{\phi_{k}\}$, the associated analysis operator $C$ is closed. Then $\Phi $ satisfies the lower frame condition,
 i.e. it is a lower semi-frame,  if and only if $C$ has closed range and is injective.
\label{lem:discranalop}
\enlem
Then, as in Proposition \ref{prop:duality},  the main result is the following:
 \begin{prop.} {\rm \cite[Prop. 3.4]{casoleli1}}
Let $\Phi =\{\phi_{k}\}$ be any total family in $\Hil$. Then $\Phi $ is a lower semi-frame if and only if there exists an upper semi-frame
$\Psi $ dual to $\Phi $, in the sense that
$$
f = \sum_{k}  \ip{\phi_{k}}{f} \, \psi_{k}   , \;\; \forall\, f\in \dom(C).
$$
\label{prop:discrduality}  
\end{prop.}
\vspace*{-5mm} In conclusion,  there is an (almost) complete symmetry between upper and lower semi-frames. 
Further results along these  lines may be found in \cite{jpa-Besseq,BSA,casoleli1} to which we refer.

\subsection{Generalization of discrete frames}
\label{subsec:rankn-discrframes}

Rank-$n$ frames were introduced in \cite[Sec.2]{jpa-contframes} in the general case of a measure space $(X,\nu)$. This consists essentially of a collection of
$n$-dimensional subspaces, one for each $x\in X$, with basis  $\{{\psi}^{i}_{x}\}, i = 1, 2, \ldots , n < \infty$, and for which
   there exist constants ${\sf m} > 0$  and ${\sf M}<\infty$  such that 
$$
{\sf m}  \norm{}{f}^2  \le \sum_{i=1}^n\int_{X}  |\ip{\psi_{x}^{i}}{f}| ^2 \, \ud \nu(x) \le {\sf M}  \norm{}{f}^2 ,  \forall \, f \in \Hil.
$$
Now, in the purely discrete case, $X$ a discrete set, this obviously reduces to an ordinary frame. Yet there are plenty of nontrivial generalizations, 
as soon as one attributes weights to the various subspaces. 

The first step is to consider \emph{weighted frames}, studied in \cite{weightedfr}. Given a set of positive weights $v_{k}>0$, the family 
$\{{\psi}_{k}, k\in \}$ is a weighted frame if
$$
{\sf m}  \norm{}{f}^2  \le \sum_{k\in \Gamma}  v_{k}^2\, |\ip{\psi_{k}}{f}| ^2   \le {\sf M}  \norm{}{f}^2 ,  \forall \, f \in \Hil.
$$
Suppose now the weights are constant by blocks of size $n$, so that one has
$$
{\sf m}  \norm{}{f}^2  \le \sum_{j\in J}  v_{j}^2 \sum_{i=1}^n |\ip{\psi_{ij}}{f}| ^2   \le {\sf M}  \norm{}{f}^2 ,  \forall \, f \in \Hil.
$$
Then, for each $j$, the family  $\{\psi_{ij}, i = 1, 2, \ldots , n \}$ is   a frame for its span, namely, the $n$-dimensional subspace
$ \Hil_{j}$.  Call $\pi_{\Hil_{j}}$ the corresponding orthogonal projection. 
Let $A_j, B_j$ be the frame bounds, and   assume  that  $A := \inf_j A_j > 0$ and $B := \sup_j B_j < \infty$.
So
$A_j \norm{}{\pi_{\Hil_{j}} f}^2 \le \sum_{i=1}^n |\ip{\psi_{ij}}{f}| ^2 \le B_j \norm{}{\pi_{\Hil_{j}} f}^2$. 
Now let ${\sf m'} = \frac{\sf m}{B}$ and ${\sf M'} = \frac{\sf M}{A}$, then we get
$$
{\sf m'}  \norm{}{f}^2  \le \sum_{j\in J} v_{j}^2 \,\norm{}{\pi_{\Hil_{j}}f}^2  \le {\sf M'}  \norm{}{f}^2 ,  \forall \, f \in \Hil.
$$

In that case, the family $\{\Hil_{j}\}_{ j\in J} $ is a \emph{frame of subspaces} with respect to the weights $\{v_{j} \} _{ j\in J} $, 
a notion introduced by Casazza and Kutyniok \cite{Casaz-kutyn}, later called \emph {`fusion frames'} (see also \cite{asgari,Casaz2} and \cite{sun}). Actually, in the general definition, the subspaces
 $\{\Hil_{j}\}_{ j\in J} $ are closed subspaces of $\Hil$, of arbitrary dimension. This structure nicely generalizes frames, in particular, it yields an associated analysis, 
synthesis and frame operator and a dual object.

Given the family $\{\Hil_{j}\}_{ j\in J}  $, build their direct sum
$$
\Hil^{\oplus} := \bigoplus_{ j\in J} \Hil_{j} = \{ \{f_{j}\}_{ j\in J }: f_{j}\in \Hil_{j}, \sum_{j\in J}\norm{}{f_{j}}^2 < \infty\}\}\, .
$$
Then one considers:
\bei
\item[(i)] The \emph{synthesis operator} $C_{W,v} : \Hil^{\oplus} \to \Hil$ defined by
$$
C_{W,v}f = \sum_{ j\in J} v_{j}\,f_{j} \, ,\; \mbox{ for all }\;f = \{f_{j}\}\in \Hil^{\oplus}.
$$
Note that the series on the r.h.s. converges unconditionnally.  
\item[(ii)] The \emph{analysis operator} $D_{W,v} = {C_{W,v}}^{^{\scriptstyle\!\!\!\!\!\ast}} : \Hil \to \Hil^{\oplus}$, which is given by
$$
D_{W,v}f = \{v_{j}\,\pi_{\Hil_{j}} f\}_{ j\in J }.
$$

\item[(iii)] The \emph{frame operator} $S_{W,v} : \Hil \to \Hil$ given, as usual, by $S_{W,v}=  {C_{W,v}}^{^{\scriptstyle\!\!\!\!\!\ast}}  \;C_{W,v} $, so that
$$
S_{W,v}f =  \sum_{ j\in J} v_{j}^2 \, \pi_{\Hil_{j}} f.
$$
\eni
Most of the standard results about ordinary frames extend to frames of subspaces, for instance:
\bei
\item[(i)]  \emph{Duality} : the dual  of $\{\Hil_{j}\} _{j\in J} $ is $\{S_{W,v}^{-1}\Hil_{j}\} _{j\in J} $. 
This is a frame of subspaces with the same weights.
\item[(ii)] \emph{Reconstruction formula:}
$$
f= \sum_{ j\in J} v_{j}^2 \, S_{W,v}^{-1}\, \pi_{\Hil_{j}} f, \; \forall \, f \in \Hil.
$$
\eni
In view of this situation, it is clear that our whole analysis of upper and lower semi-frames made in Section \ref{subsec:unbdd-discrframe} extends as well.

Further generalizations have been introduced, namely \emph{g-frames} \cite{khosravi,sun}. 
A parallel analysis can be made, but we will refrain from doing it here.

\section*{Acknowledgements}

The authors would like to thank D. Stoeva and O. Christensen for helpful comments and suggestions.  
This work was partly supported by the WWTF project MULAC (`Frame Multipliers: Theory and Application in Acoustics', MA07-025).
The first author acknowledges gratefully the hospitality of  the Acoustics Research Institute, Austrian Academy of Sciences,  Vienna, 
and so does the second author towards the  Institut de Recherche en Math\'ematique et  Physique, Universit\'e catholique de Louvain.

\end{document}